\documentclass[a4paper,11pt]{article}
\pdfoutput=1

\usepackage{etex}
\usepackage{fullpage} 
\usepackage{lmodern}
\usepackage[protrusion=true,expansion=true]{microtype}

\usepackage{enumerate}

\usepackage[utf8]{inputenc}
\usepackage[english]{babel}
\usepackage[T1]{fontenc}
\usepackage{amssymb,amsmath,amsthm}
\usepackage{url}
\usepackage{graphicx}
\usepackage{mathtools}
\usepackage{subfigure}

\usepackage{pst-all}

\usepackage{tikz-cd}
\usepackage{amscd}

\usepackage{bbm}
\usepackage{accents}
\usepackage{todonotes}
\usepackage{longtable}
\usepackage{tabulary}

\newcommand{\RR}{\mathbb{R}}
\newcommand{\CC}{\mathbb{C}}
\newcommand{\ZZ}{\mathbb{Z}}
\newcommand{\NN}{\mathbb{N}}

\newcommand{\PP}{\mathbb{P}}
\newcommand{\CP}{\mathbb{CP}}

\newcommand{\kk}{\Bbbk}

\newcommand{\HH}{\mathbb H}

\newcommand{\mcA}{\mathcal A}

\newcommand{\mcC}{\mathcal C}

\newcommand{\mcM}{\mathcal M}
\newcommand{\mcN}{\mathcal N}

\newcommand{\mcP}{\mathcal P}

\newcommand{\defeq}{\vcentcolon=}

\newcommand{\id}{\operatorname{id}}

\newcommand{\Spec}{\operatorname{Spec}}

\newcommand{\Span}{\operatorname{span}}

\newcommand{\ind}{\operatorname{ind}}
\newcommand{\lcm}{\operatorname{lcm}}
\newcommand{\Ham}{\operatorname{Ham}}
\newcommand{\Sp}{\operatorname{Sp}}
\newcommand{\diag}{\operatorname{diag}}

\newcommand{\im}{\operatorname{im}}

\newcommand{\dd}{\partial}

\newcommand{\RS}{\mathrm{RS}}

\newcommand{\product}{\mathrm{product}}
\newcommand{\Morse}{\mathrm{Morse}}

\newcommand{\incl}{\mathrm{incl}}

\DeclarePairedDelimiter{\ceil}{\lceil}{\rceil}
\DeclarePairedDelimiter{\floor}{\lfloor}{\rfloor}

\DeclarePairedDelimiterX{\shiftedinterval}[2]{[}{)}{#1, #2}

\newtheorem{thm}{Theorem}[section]
\newtheorem{prop}[thm]{Proposition}
\newtheorem{cor}[thm]{Corollary}
\newtheorem{lem}[thm]{Lemma}
\newtheorem{claim}[thm]{Claim}

\theoremstyle{definition}
\newtheorem{defn}[thm]{Definition}
\newtheorem{example}[thm]{Example}
\newtheorem{rmk}[thm]{Remark}

\title{Periodic Reeb flows and products in symplectic homology}
\author{Peter Uebele}
\date{}

\begin{document}

\maketitle

\begin{abstract}
 In this paper, we explore the structure of Rabinowitz--Floer homology $RFH_*$ on contact manifolds whose Reeb flow is periodic (and which satisfy an index condition such that $RFH_*$ is independent of the filling). The main result is that $RFH_*$ is a module over the Laurent polynomials $\ZZ_2[s,s^{-1}]$, where $s$ is the homology class generated by a principal Reeb orbit and the module structure is given by the pair-of-pants product. In most cases, this module is free and finitely generated.
\end{abstract}

\section{Introduction}

Symplectic homology, as introduced in \cite{Floer-Hofer, Viterbo_part_1}, is a generalization of Floer theory to non-compact symplectic manifolds (or compact symplectic manifolds with boundary). 
As such, it has not only an additive structure (chain groups and a differential), but also other algebraic operations, coming from counting Riemann surfaces with an arbitrary number of positive and negative punctures. Most notably, there is a (commutative and associative) product called pair-of-pants product and a unit, giving symplectic homology the structure of a commutative unital ring.

On the other hand, symplectic homology is very hard to compute already on the additive level, mainly because the differential is defined by counting solutions to a certain PDE involving, among other things, the choice of a generic almost complex structure. The same difficulty applies for computations of the product structure. Therefore, the ring structure of symplectic homology is known only in few examples. Known examples include:
\begin{itemize}
 \item Subcritical Stein manifolds \cite{Cieliebak_handle}, where symplectic homology vanishes,
 \item Cotangent bundles \cite{Viterbo_part_2, Abb_Schwarz}, where symplectic homology is isomorphic to the homology of the loop space with the Chas--Sullivan product,
 \item Negative line bundles \cite{Ritter_neg_line}, where symplectic homology is related to Gromov--Witten theory and quantum cohomology.
\end{itemize}

This paper attempts to apply the techniques of \cite{Ritter_neg_line} to the more general case of contact manifolds with periodic Reeb flow. For this to make sense, we need a notion of symplectic homology that is an invariant of contact manifolds, not the fillings. One such notion could be positive symplectic homology, but this has the drawback of not carrying a unital product. Instead, we will use the $\bigvee$-shaped symplectic homology $\check{SH}_*$ from \cite{Ciel_Fra_Oan}, which is isomorphic to Rabinowitz Floer homology $RFH_*$ by \cite[Theorem~1.5]{Ciel_Fra_Oan}. The following proposition is a variation of \cite[Theorem~1.14]{Ciel_Fra_Oan} and is proved using SFT-compactness and neck-stretching techniques.

For simplicity, we assume throughout the introduction that $\pi_1(\Sigma)=0$ in order to avoid ambiguity in the grading. See Remark~\ref{rmk_grading} for a discussion of this assumption. Moreover, we always use $\ZZ_2$-coefficients, unless stated otherwise.

\begin{prop} \label{prop_intro_indep}
 Let $(\Sigma, \alpha)$ be a $(2n-1)$-dimensional contact manifold with $\pi_1(\Sigma)=0$, $c_1(\Sigma)=0$ and the condition
 \begin{equation} \label{eq_intro_index_cond}
  \mu_{CZ}(c) > 4-n \qquad \text{for all closed Reeb orbits $c$},
 \end{equation}
 on the Conley--Zehnder indices. Then, $\check{SH}_*(\Sigma)$ can be defined in the symplectization $\RR_+\times \Sigma$, without reference to any symplectic filling of $\Sigma$. 
 Moreover, if there exists a Liouville filling $W$ such that $c_1(W)=0$, then $\check{SH}_*(\Sigma) \cong \check{SH}_*(W)$.
\end{prop}

By \cite{Ciel_Oan}, $\check{SH}_*(W)$ carries a commutative, unital product, just like $SH_*(W)$. Furthermore, for contact manifolds satisfying a stronger index condition, namely
\begin{equation} \label{eq_intro_index_cond_2}
 \mu_{CZ}(c) > 3 \qquad \text{for all closed Reeb orbits $c$},
\end{equation}
the product can also be defined without reference to a filling.

Our main tool to get structural results for $\check{SH}_*(\Sigma)$, which builds upon ideas from \cite{Seidel_pi_1, Ritter_neg_line}, is to study the action of a loop of Hamiltonian diffeomorphisms
\[
 g \colon S^1 \longrightarrow \Ham(\RR_+\times \Sigma, d(r\alpha)), \qquad t\mapsto g_t
\]
on $\check{SH}_*(\Sigma)$. This action is defined by $\gamma(t)\mapsto g_t\cdot \gamma(t)$ on the level of generators, and similarly by $u\mapsto g_t\cdot u$ on the Floer cylinders counted by the differential. In this way, $g_t$ defines an isomorphism
\begin{equation} \label{eq_intro_iso}
 S_g \colon \check{SH}_*(\Sigma) \stackrel{\cong}{\longrightarrow} \check{SH}_{*+2I(g)}(\Sigma),
\end{equation}
where $I(g)$ is a Maslov index depending only on the loop $g_t$. In this paper, we are mainly interested in the example where $g_t$ is given by the Reeb flow on $\Sigma$, which is always possible if the Reeb flow is periodic (with the period normalized to one).
In most cases, this loop of Hamiltonian diffeomorphisms does not extend to a symplectic filling of $\Sigma$, hence the need to work in the symplectization.

The isomorphism \eqref{eq_intro_iso} does not preserve the product, but instead satisfies the relation
\[
 S_g (x\cdot y) = S_g(x) \cdot y.
\]
In particular, if we take $x$ to be the unit we get $S_g(y) = s\cdot y$, where $s\defeq S_g(1)$ is the principal orbit of $(\Sigma, \alpha)$. Furthermore, by taking the loop $g$ in the reverse direction, we get the element $s^{-1}$ inverse to $s$. 

\begin{thm}
\label{thm_intro_module}
 Let $(\Sigma, \alpha)$ be a contact manifold with periodic Reeb flow (with the period normalized to one) satisfying \eqref{eq_intro_index_cond_2} and $\pi_1(\Sigma)=0$. Then, $\check{SH}_*(\Sigma)$ is a module over the ring of Laurent polynomials $\ZZ_2[s,s^{-1}]$, with multiplication given by the pair-of-pants product
 \[
  (s^k, x) \mapsto S_g^k(x) = s^k\cdot x.
 \]
 If $I(g)\neq 0$ this module is free and finitely generated. By contrast, if $I(g)=0$ then $\check{SH}_*(\Sigma)$ is a free module (i.e.\ a vector space) over the field of Laurent series $\ZZ_2((s^{-1}))$.\footnote{
 The notation $\ZZ_2((s^{-1}))$ stands for the field of semi-infinite Laurent series of the form $\sum_{j=-\infty}^N c_j s^j$, i.e.\ the powers of $s$ can go to $-\infty$.}
 
 In both cases, the dimension of this module is bounded from above by $\dim_{\ZZ_2} \left(\check{SH}^{\shiftedinterval{0}{1}}(\Sigma)\right)$.
\end{thm}

To put this result into context, recall that $\check{SH}_*(\Sigma)$ is usually not finitely generated as a $\ZZ_2$-vector space, so only the product gives a finite algebraic structure. 
Furthermore, Theorem~\ref{thm_intro_module} gives some product computation that would be very difficult to prove directly.
In examples, however, it turns out that there can be further relations between the generators of the module, so Theorem~\ref{thm_intro_module} does not reveal the full ring structure of $\check{SH}_*(\Sigma)$.

While the index conditions \eqref{eq_intro_index_cond} and \eqref{eq_intro_index_cond_2} are quite restrictive, both of them can be relaxed if $\Sigma$ admits a Liouville filling $W$ with $c_1(W)=0$. Then, indeed, \eqref{eq_intro_index_cond} can be replaced by $\mu_{CZ}(c) > 3-n$ for all Reeb orbits $c$. Moreover, if in addition $I(g)\neq 0$, then the conclusion of Theorem~\ref{thm_intro_module} also holds under the weaker assumption that $\mu_{CZ}(c) > 3-n$ for all Reeb orbits, see Proposition~\ref{prop_weaker_assumption}.

An important class of examples for Theorem~\ref{thm_intro_module} are Brieskorn manifolds. While the additive structure of symplectic homology has been studied in \cite{Fauck, KvK, Uebele}, the product structure remains largely unexplored (except for some special examples, see Section~\ref{sec_comparison}). Apart from giving some product computations, Theorem~\ref{thm_intro_module} (or equation \eqref{eq_intro_iso}) implies that $\check{SH}_*(\Sigma)$ fulfills the periodicity
\[
 \check{SH}_*(\Sigma) \cong \check{SH}_{*+2I(g)}(\Sigma).
\]
While this periodicity is easy to establish for Brieskorn manifolds on the chain level (e.g.\ with the chain complex as in \cite{Fauck}), it is far less obvious on homology.

Finally, Theorem~\ref{thm_intro_module} can also be used to get some information about the usual symplectic homology $SH_*(W)$ of a Liouville filling $W$ of $\Sigma$. The long exact sequence from \cite{Ciel_Fra_Oan} gives a map
\[
 f\colon SH_*(W) \longrightarrow \check{SH}_*(\Sigma),
\]
whose kernel is a subset of the negative symplectic homology $SH_*^-(W)$. In fact, $f$ is a ring homomorphism (see Lemma~\ref{lem_f_intertwines_product} or \cite[Theorem~10.2(e)]{Ciel_Oan}), hence $SH_*(W)/\ker(f)$ is a ring and maps injectively to $\check{SH}_*(\Sigma)$. It turns out that, with the right choice of module generators, the image of $SH_*(W)/\ker(f)$ in $\check{SH}_*(\Sigma)$ is the subset of elements with non-negative powers of $s$.

\begin{cor} \label{cor_intro_SH}
 Let $\Sigma$ be as in Theorem~\ref{thm_intro_module} and $W$ a Liouville filling of $\Sigma$ with $c_1(W)=0$.
 Then, $SH_*(W)/\ker(f)$ is a free and finitely generated module over $\ZZ_2[s]$. In particular, $SH_*(W)$ is finitely generated as a $\ZZ_2$-algebra.
\end{cor}

This text is organized as follows. In Section~\ref{sec_product_structure}, we recall the definition of the pair-of-pants product and $\bigvee$-shaped symplectic homology.
Section~\ref{sec_Ham_diff} contains the main results and their proofs. After recalling some facts on the action of Hamiltonian diffeomorphisms in general, we show in Section~\ref{sec_no_filling} that in most cases, the Reeb flow does not extend to a Hamiltonian diffeomorphism on any symplectic filling, hence the need to work in the symplectization. Section~\ref{sec_def_on_symplectization} gives precise arguments how this is possible. After that, we can establish the $\ZZ_2[s,s^{-1}]$-module structure of $\check{SH}_*(\Sigma)$, though still without relating it to the pair-of-pants product. This relation is finally proven in Sections~\ref{sec_homot_inv} and \ref{sec_application_product}, while Section~\ref{sec_back_to_usual_SH} proves Corollary~\ref{cor_intro_SH}.

In Section~\ref{sec_Brieskorn_examples}, we apply the theorems from Section~\ref{sec_Ham_diff} to the example of Brieskorn manifolds, where the chain groups of $\check{SH}_*(\Sigma)$ can be computed explicitly. This section also contains comparisons to specific examples where the ring structure of symplectic homology is known, namely cotangent bundles of spheres and the $A_2$-surface singularities. The upshot of these comparisons is that the results of this paper are confirmed in these examples, but they reveal only a part of the product structure on symplectic homology.

\vspace{0.3cm}
\noindent 
\emph{Acknowledgements.} I would like to thank my advisor Kai Cieliebak for his continued support and guidance on the subject.
I also want to thank the anonymous referee for his careful reading and for pointing out some inaccuracies in an earlier version of this paper.

\section{Product structure on Hamiltonian Floer homology} \label{sec_product_structure}

\subsection{General definition} \label{sec_product_def}

The product in Hamiltonian Floer theory always involves a count of pairs-of-pants between three Hamiltonian orbits, although the precise definition varies slightly in the literature. Here, we will follow the approach from \cite{Abouzaid}. Let $\mcP \defeq \PP^1\setminus \{0,1,\infty\}$ be the Riemann sphere with three punctures, two of which are called \emph{positive} (or inputs) and one is called \emph{negative} (or the output). Fix parametrizations $[0,\infty) \times S^1$ near the positive punctures and $(-\infty, 0] \times S^1$ near the negative puncture, called \emph{cylindrical ends}.

Throughout this text, $\widehat W$ denotes the completion of a Liouville domain $W$ with boundary $\Sigma=\dd W$ and we take coefficients in $\ZZ_2$.
Given Hamiltonians $H_0, H_1, H_2 \in \mcC^\infty(\widehat W)$, almost complex structure $J_0, J_1, J_2$ and $1$-periodic orbits $\gamma_0, \gamma_1, \gamma_2$ of the Hamiltonians, respectively, we want to define the product
\begin{equation} \label{eq_product_first_def}
 HF(H_1, J_1) \times HF(H_2, J_2) \to HF(H_0, J_0).
\end{equation}

To define this product, we need the following data:
\begin{itemize}
 \item A Hamiltonian $H_\mcP$, parametrized by the pair-of-pants surface $\mcP$, such that $H_\mcP(s,t,x) = H_i(t,x)$ in the parametrization near the puncture $z_i$.
 \item An almost complex structure $J_\mcP$, parametrized by $\mcP$, such that $J_\mcP(s,t,x) = J_i(t,x)$ in the parametrization near the puncture $z_i$.
 \item A one-form $\beta \in \Omega^1(\mcP)$ which restricts to $dt$ in the parametrizations near the punctures.
\end{itemize}

Assume that $J_\mcP$ is convex near infinity, i.e.\ outside a compact set of $\widehat W$, 
\[
 d r \circ J_\mcP(s,t,x) = -e^f \lambda,
\]
where $r$ is the radial coordinate, $\lambda$ is a primitive of the symplectic form and $f$ is any smooth function.
Moreover, assume that the Hamiltonians $H_0, H_1, H_2$ are linear at infinity with slopes $b_0, b_1, b_2 \geq 0$ and $H_\mcP$ is linear at infinity with slope function $b_\mcP\colon \mcP \to \RR_+$. Then we require (for compactness of the moduli spaces below) that
\begin{equation} \label{eq_1form_compactness}
 d(b_\mcP \beta) \leq 0.
\end{equation}
By \cite[Exercise~2.3.4]{Abouzaid}, it is possible to choose $\beta$ and $H_\mcP$ such that \eqref{eq_1form_compactness} is satisfied if and only if $b_0\geq b_1+b_2$. Now, we define the moduli space of pairs-of-pants
\[
 \mcM(\gamma_1, \gamma_2, \gamma_0; \beta, H_\mcP, J_\mcP)
\]
as the set of smooth maps $u\colon \mcP \to \widehat W$ which converge to $\gamma_1, \gamma_2$ at the positive punctures and to $\gamma_0$ at the negative puncture and satisfy the Floer equation
\begin{equation} \label{eq_Floer_beta}
 \left(du - X_{H_\mcP} \otimes \beta\right)^{0,1} = \frac{1}{2} \left((du - X_{H_\mcP}\otimes \beta) + J \circ (du - X_{H_\mcP}\otimes \beta) \circ j\right) = 0.
\end{equation}
For a generic choice of $H_\mcP$ and $J_\mcP$, this moduli space is a smooth manifold of dimension
\[
 \dim(\mcM(\gamma_1, \gamma_2, \gamma_0; \beta, H_\mcP, J_\mcP)) = \mu(\gamma_1) + \mu(\gamma_2) - \mu(\gamma_0) - n,
\]
where $\mu = \mu_{CZ}$ denotes the Conley--Zehnder index. Moreover, there is a suitable compactification by adding lower-dimensional strata. In particular, for $\mu(\gamma_0) = \mu(\gamma_1) + \mu(\gamma_2) + n$, the moduli space is a finite set of points. Hence, we can define the product of $\gamma_1$ and $\gamma_2$ as
\[
 \gamma_1 \cdot \gamma_2 = \sum_{\substack{\gamma_0 \\ \mu(\gamma_0) = \mu(\gamma_1) + \mu(\gamma_2) - n}} \#_2 [\mcM(\gamma_1, \gamma_2, \gamma_0; \beta, H_\mcP, J_\mcP)] \; \gamma_0,
\]
giving the definition of \eqref{eq_product_first_def}. By \cite[Section~2.3.6]{Abouzaid}, this product behaves well with respect to continuation maps. Hence, taking direct limits on the Hamiltonians, it induces a product
\[
 SH_k(W) \times SH_\ell(W) \to SH_{k+\ell-n}(W).
\]
It turns out that this product is associative and graded commutative, although this is not obvious from the definition. Also, there is an element of $SH$ acting as a unit of this product, namely the image of the generator of $H^0(W)$ under the map $H^*(W)\cong SH^-_{n-*}(W) \to SH_{n-*}(W)$. Hence, it gives $SH$ the structure of a unital, graded-commutative ring.

\subsection{Grading and action filtration}

By definition, the pair-of-pants product has degree $-n$ in the usual grading. In order to have a product of degree zero, it can be convenient to switch to the ``product grading''
\[
 \mu_\product \defeq \mu - n,
\]
as we will do in Section~\ref{sec_Brieskorn_examples}.\footnote{
 This grading convention is consistent with \cite{Ciel_Oan, Ciel_Fra_Oan}. Note that another common convention is to use the negative of this grading.
}

\begin{rmk} \label{rmk_action_filtr}
 Although we will not need this, let us recall how the pair-of-pants product on symplectic homology respects the action filtration. For this purpose, it is convenient to use a slightly different definition of the product, in which the Hamiltonians $H_1,H_2$ and $H_0$ are positive multiples of a common Hamiltonian $H$, see e.g.\ \cite{Ritter_tqft}. (The induced product on $SH$ is still the same.) Then, by \cite[Section~16.3]{Ritter_tqft}, it holds that
 \[
  \mcA_{H_0}(\gamma_1 \cdot \gamma_2) \leq \mcA_{H_1}(\gamma_1) + \mcA_{H_2}(\gamma_2) ,
 \]
 As a consequence, the product restricts to a map
\[
\cdot \colon SH^{[a,b)} \times SH^{[a',b')} \to SH^{[\max\{a+b', a'+b\}, b+b')}, 
\]
where on the right hand side, it is necessary to divide out all generators with action less than $\max\{a+b', a'+b\}$ to make the map well-defined. For example, one does not get a product on the whole positive symplectic homology, but one can define maps
\begin{equation} \label{eq_prod_on_pos_SH}
 SH^{[\delta, b)} \times SH^{[\delta, b)} \to SH^{[b+\delta, 2b)}
\end{equation}
that contain a part of the information of the product on $SH$.
\end{rmk}

\subsection{$\bigvee$-shaped symplectic homology and its product structure} \label{sec_V_shaped}

For the purposes of this text, it will be important to have a version of symplectic homology that is defined on the symplectization of a contact manifold, without reference to a symplectic filling. This is definitely not possible for the usual $SH$, as even some of its generators live in the filling (indeed, its negative part $SH^-$ is isomorphic to the singular cohomology of the filling). 

Its positive part $SH^+$ is, under favorable conditions, independent of the filling. However, $SH^+$ does not have a product, see Remark \ref{rmk_action_filtr}. One might try to use the ``partial products'' from \eqref{eq_prod_on_pos_SH} instead, but this has the drawback that there is no unit, which is needed in some arguments below.

\begin{figure}[h] 
 \centering
 \def\svgwidth{0.35\textwidth}
 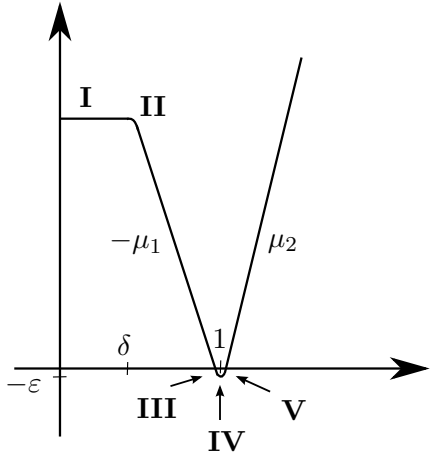
 \caption{A Hamiltonian used to define $\check{SH}$}
 \label{fig:V_shaped_H}
\end{figure}

The solution to this is to use the $\bigvee$-shaped symplectic homology $\check{SH}$ of \cite{Ciel_Fra_Oan}. Let us quickly recall how this homology theory is constructed: Take a Hamiltonian as in Figure~\ref{fig:V_shaped_H} with $\mu_1, \mu_2 \notin \Spec(\Sigma, \alpha)$. (In \cite{Ciel_Fra_Oan}, $\mu_1=\mu_2$, but it causes no problems to have different values.) The $1$-periodic orbits are concentrated in the areas (I) to (V). However, as explained in \cite[Proposition~2.9]{Ciel_Fra_Oan}, the orbits in (I) and (II) are excluded by their action. Indeed, given an action window $(a,b)$, one can choose the constants $\mu_1, \mu_2, \delta$ and $\varepsilon$ such that all generators with action in $(a,b)$ are of the following types:
\begin{itemize}
 \item Nonconstant orbits in (III), coming from negatively parametrized Reeb orbits with action greater than $a > -\mu_1$.
 \item Constant orbits in (IV), coming from the singular cohomology of $\Sigma$.
 \item Nonconstant orbits in (V), coming from positively parametrized Reeb orbits with action less than $b < \mu_2$.
\end{itemize}
Then, define
\begin{equation}\label{eq_lim_SH_1}
 \check{SH}_k^{(a,b)}(\widehat W) \defeq \varinjlim_H HF_k^{(a,b)}(H)
\end{equation}
as the direct limit as $\mu_1, \mu_2 \to \infty$, and define
\begin{equation}\label{eq_lim_SH_2}
 \check{SH}_k(\widehat W) \defeq \varinjlim_b \varprojlim_a\check{SH}_k^{(a,b)}(\widehat W),
\end{equation}
where the limits mean $b\to\infty$ and $a\to -\infty$, respectively.

By \cite[Theorem~1.5]{Ciel_Fra_Oan}, $\check{SH}(\widehat W)$ is isomorphic to the Rabinowitz Floer homology of $W$. Moreover, the positive part $\check{SH}^{(0,\infty)}(\widehat W)$ is isomorphic to the usual positive symplectic homology $SH^+(\widehat W)$, while $\check{SH}^{(-\epsilon,\epsilon)}(\widehat W)$ (for $\epsilon>0$ sufficiently small) is isomorphic to the singular cohomology of $\Sigma$.

As explained in \cite{Ciel_Oan}, the product structure described in Section~\ref{sec_product_def} also lives on $\check{SH}(\widehat W)$. Similarly to the usual symplectic homology, this product has a unit, coming from the generator of $H^0(\Sigma)$. 

\section{$S^1$-actions by loops of Hamiltonian diffeomorphisms} \label{sec_Ham_diff}

\subsection{Recollections from the closed case} \label{sec_S1_generalities}

In this section, we recall some facts from \cite{Seidel_pi_1} on the action of a loop of Hamiltonian diffeomorphisms on Floer homology on a closed symplectic manifold $(M,\omega)$. Let
\[
 g\colon S^1 = \RR/\ZZ \to \Ham(M,\omega), \qquad t\mapsto g_t
\]
be a loop of Hamiltonian diffeomorphisms based at $g_0=\id$. Denote by $K^g\colon S^1\times M \to \RR$ a Hamiltonian function that generates $g$, i.e.\ $\dd_t(g_t \cdot)=X_{K^g}(t,g_t\cdot)$.

In this text, we will only work with manifolds $(M,\omega)$ that satisfy $c_1(M)|_{\pi_2(M)} = 0$ and $\omega|_{\pi_2(M)}=0$ (actually, in the non-closed case, $\omega$ will be an exact form). Therefore, the grading and the action functional will be well-defined and we do not need any cover of the loop space or Novikov coefficients (see \cite[Section~2.4]{Ritter_neg_line}).

The loop $g$ acts on the loop space $C^\infty(S^1, M)$ by
\[
 (g\cdot \gamma)(t) = g_t(\gamma(t)).
\]
Define the pullback $(g^*H, g^*J)$ of a pair of Hamiltonian $H$ and almost complex structure $J$ as
\[
 (g^*H_t)(x) = H_t(g_t(x)) - K_t^g(g_t(x)), \qquad g^*J_t = dg_t^{-1} \circ J_t \circ dg_t.
\]
Similarly, define the pushforward $(g_*H, g_*J)$ as
\[
 (g_*H_t)(x) = \left((g^{-1})^*H_t \right)(x) = H_t(g_t^{-1}(x)) + K_t^g(g_t^{-1}(x)), \qquad g_*J_t = (g^{-1})^*J_t = dg_t \circ J_t \circ dg_t^{-1}.
\]

\begin{lem} \label{lem_g_properties}
The action of $g$ has the following properties:
\begin{enumerate}
 \item $g^*(d\mcA_H) = d\mcA_{g^*H}$, where $\mcA_H(\gamma) = \int_{D^2} \bar{\gamma}^*(\omega) - \int_{S^1} H_t(\gamma(t)) dt$ is the usual symplectic action functional.
 Equivalently, $\mcA_{g^*H} = g^*\mcA_H$ up to a constant (depending on the choice of additive constant for $K^g$).
 \item $1$-periodic orbits of $H$ correspond bijectively to $1$-periodic orbits of $g_*H$ via $x \mapsto g \cdot x$
 \item Floer trajectories satisfy the bijective correspondence
 \[
  \mcM(\gamma_+,\gamma_-; H,J) \stackrel{\cong}{\longrightarrow} \mcM(g\cdot \gamma_+, g\cdot \gamma_-; g_*H, g_*J), \qquad u \mapsto g \cdot u,  
 \]
 and similarly for the moduli spaces appearing in the continuation maps.
\end{enumerate}
\end{lem}

See \cite[Section~4]{Seidel_pi_1} for the proof of Lemma \ref{lem_g_properties}.
As for the grading, the Maslov index $I(g)\in \ZZ$ is defined as follows. For any contractible loop $\gamma\in C^\infty(S^1, M)$, choose a filling disk, which induces a symplectic trivialization
\[
 \tau_\gamma \colon \gamma^*(TM) \to S^1\times (\RR^{2n}, \omega_0)
\]
of the pullback bundle $\gamma^*(TM)$. By \cite[Lemma~2.2]{Seidel_pi_1}, $g \cdot \gamma$ is also contractible. Thus, $g(t)$ induces a loop of symplectomorphisms $\ell(t)\in \Sp(2n,\RR)$ by
\[
 \ell(t) = \tau_{g\gamma}(t) \circ dg_t(\gamma(t)) \circ \tau_\gamma(t)^{-1}.
\]
Define the Maslov index $I(g)\defeq \deg(\ell)$, where $\deg\colon H_1(\Sp(2n,\RR)) \to \ZZ$ is the isomorphism induced by the determinant on $U(n)\subset Sp(2n,\RR)$.
By the assumption that $c_1(M)|_{\pi_2(M)} = 0$, this index is independent of the choice of filling disks. In fact, it is also independent of $\gamma$ and only depends on the homotopy class of $g_t$ in $\pi_0(\Ham(M,\omega))$.
So 
\[
 \mu(g \cdot \gamma) = \mu(\gamma) + 2I(g),
\]
by one of the axioms of the Conley--Zehnder index.

\begin{cor} \label{cor_S_g_iso}
 The loop $g_t$ induces a map on Floer homology
 \[
  S_g \colon HF_*(H) \to HF_{*+2I(g)}(g_*H).
 \]
 As $g^{-1}$ gives the inverse map, $S_g$ is in fact an isomorphism.
\end{cor}

The following proposition gives two further properties, whose proofs are a bit more involved (see \cite[Sections~5 and~6]{Seidel_pi_1}):

\begin{prop} 
 \begin{enumerate}
  \item If $g_t$ and $\tilde g_t$ are homotopic through a homotopy of loops of Hamiltonian diffeomorphisms $g_t^r$ with $g_0^r = \id$ for all $r$, then 
  \[
   S_g = S_{\tilde g} \colon HF_*(M, \omega) \to HF_{*+2I(g)}(M, \omega).
  \]
  \item 
  The isomorphism $S_g$ and the pair-of-pants product $\cdot$ fulfill the relation
  \[
   S_g(x\cdot y) = S_g(x) \cdot y.
  \]
 \end{enumerate}
 \label{prop_less_basic}
\end{prop}

\subsection{$S^1$-actions by Hamiltonian loops on $\widehat W$} \label{sec_no_filling}

All of the statements of Section~\ref{sec_S1_generalities}, including Proposition~\ref{prop_less_basic}, admit a rather straightforward generalization to symplectic homology, provided that the filling $W$ admits a Hamiltonian $S^1$-action. This generalization has been worked out by Ritter in \cite{Ritter_neg_line}. Unfortunately, in many examples, one has a suitable $S^1$-action (e.g.\ by the Reeb flow) only on the contact manifold (and hence on its symplectization), but not on the filling. Indeed, the following lemma shows that in many cases, the $S^1$-action cannot be extended to a Liouville filling. 

\begin{lem}[Corollary from \cite{Ritter_neg_line}] \label{lem_no_S1_filling}
Let $\Sigma$ be a contact manifold with periodic Reeb flow and $W$ a Liouville filling of $\Sigma$ (with arbitrary first Chern class) such that $SH_*(W)$ does not vanish. Then, the $S^1$-action on $\Sigma$ by the Reeb flow does not extend to an $S^1$-action of Hamiltonian diffeomorphisms on $W$. 
\end{lem}

\begin{proof}
 This lemma follows immediately from \cite[Corollary\ 2]{Ritter_neg_line}. The reasoning goes like this:
 
 If the $S^1$-action does extend, we get the map $S_g$. From this map, one can construct an endomorphism $r_g$ from the quantum cohomology $QM^*(W)$ to itself (basically as the composition of $S_g$ with a continuation map and the isomorphisms between $QM^*(W)$ and Floer homology of the zero Hamiltonian). By \cite[Theorem\ 1]{Ritter_neg_line}, the symplectic homology of $W$ is a quotient of $QM^*(W)$ by the generalized $0$-eigenspace of $r_g$, i.e.
 \[
  SH_*(W) \cong QH^*(W) / \ker(r_g^k)
 \]
 for $k$ sufficiently large.
 In our context, $W$ is exact, so its quantum cohomology reduces to its ordinary cohomology with the cup product. For this case, \cite{Ritter_neg_line} shows that $r_g$ is nilpotent, hence $SH_*(W)=0$.
\end{proof}

Note that the assumption $SH \neq 0$ is necessary, since otherwise, the ball in $\CC^n$ would provide a counterexample.

Lemma~\ref{lem_no_S1_filling} can be applied directly to Brieskorn manifolds. 
Indeed, by \cite[Theorem~6.3]{KvK}, the standard filling $W$ of any Brieskorn manifold $\Sigma(a)$ (with $a_j\geq 2$ for all $j$) fulfills $SH(W)\neq 0$. Hence, the $S^1$-action $g_t$ does not extend to $W$.

\begin{rmk}
 It is instructive to consider the example $\Sigma(2,\ldots, 2)$, which is contactomorphic to the unit cotangent bundle $S^*S^n$ of $S^n$. The $S^1$-action by the Reeb flow agrees with the geodesic flow for the standard Riemannian metric on $S^n$. While the geodesic flow extends to the filling $D^*S^n$, the period varies, so this does not give an $S^1$-action. On the other hand, the normalized geodesic flow is an $S^1$-action, but it does not extend across the zero-section in $D^*S^n$.
\end{rmk}

Because of this non-existence, the only way one can hope to apply the results of Section~\ref{sec_S1_generalities} to Brieskorn manifolds is to use a version of symplectic homology that can be defined purely on the symplectization. In the next section, we show that this is possible with $\check{SH}$ in many cases. 

\subsection{Defining $\check{SH}$ on $\RR_+ \times \Sigma$} \label{sec_def_on_symplectization}

We take the model $(\RR_+ \times \Sigma, \omega = d(r\alpha))$ for the symplectization. An $\omega$-compatible almost complex structure $J_t$ is called \emph{SFT-like} if it satisfies
\begin{itemize}
 \item $J_t (r\dd_r) = R_\alpha$, where $R_\alpha$ denotes the Reeb vector field.
 \item $J_t$ preserves the contact distribution $\xi=\ker(\alpha)$.
 \item $J_t$ is invariant under translations $r\mapsto e^c r$ for $c\in\RR$.
\end{itemize}
Now, fix a Hamiltonian $H = H_{\mu_1,\mu_2}$ as in Figure~\ref{fig:V_shaped_H} and an $\omega$-compatible almost complex structure $J_t$ which is SFT-like near the negative end of the symplectization.

\begin{lem} \label{lem_cyl_cpt}
Assume $c_1(\Sigma)=0$ and $\mu_{CZ}(c) > 3-n$ for all contractible Reeb orbits $c$.
Let $\gamma_+, \gamma_-$ be two Hamiltonian orbits in the part where $H$ is convex with $\mu(\gamma_+) - \mu(\gamma_-)=1$. Then, the moduli space $\mcM^{\RR_+ \times \Sigma}(\gamma_+, \gamma_-; H, J)$ is compact, i.e.\ the Floer cylinders do not escape to the negative end of the symplectization.
\end{lem}

\begin{figure}[h] 
\begin{minipage}{0.6\textwidth}
 \centering
 \def\svgwidth{\textwidth}
 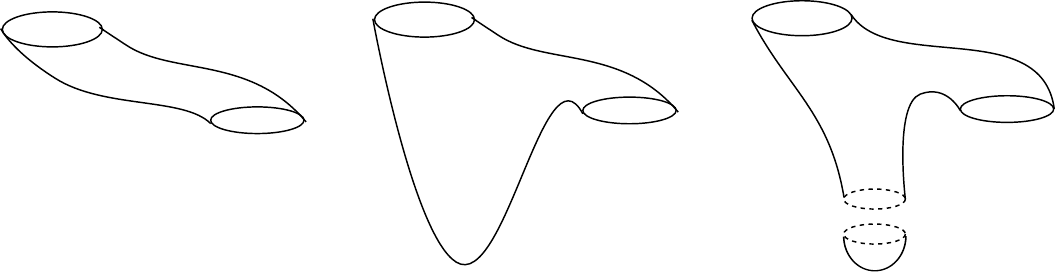
 \caption{Possible breaking of cylinders. Hamiltonian orbits are represented by continuous lines, Reeb orbits by dashed lines.}
 \label{fig:cylinder_breaking}
\end{minipage}
\hfill
\begin{minipage}{0.35\textwidth}
 \centering
 \def\svgwidth{0.5\textwidth}
 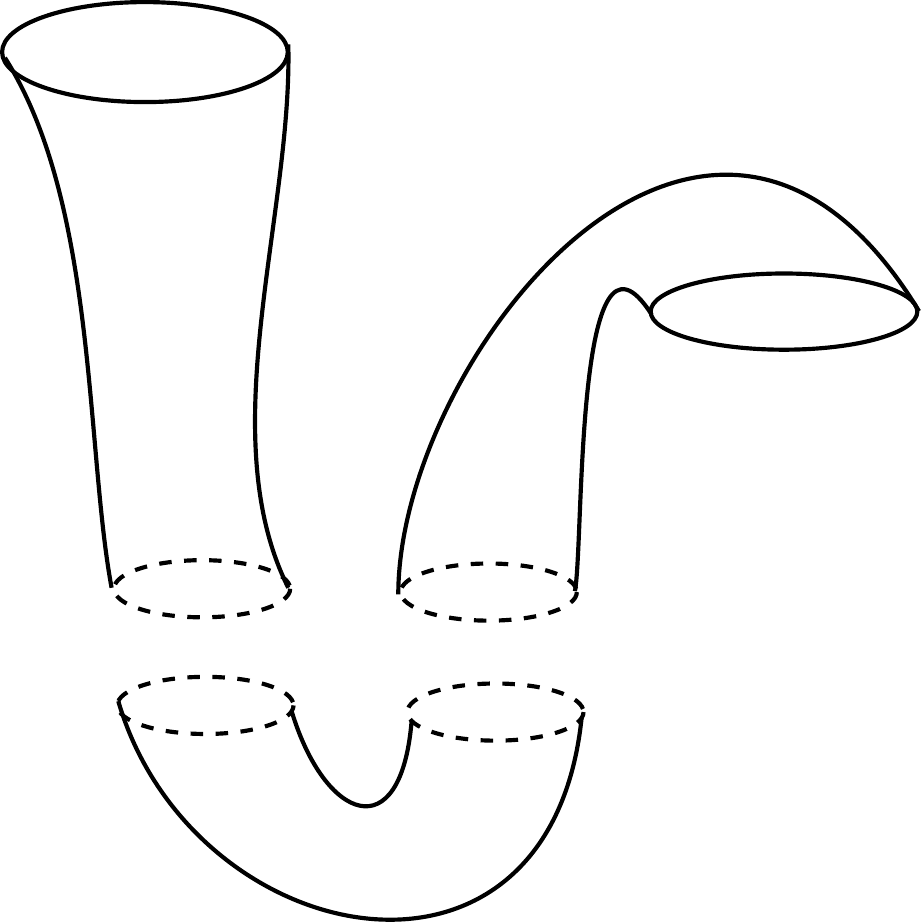
\caption{Such a breaking cannot occur, due to the maximum principle}
\label{fig:cylinder_not_possible}
\end{minipage}
\end{figure}

\begin{proof}
Assume that there exists a sequence of Floer cylinders $u_j \in \mcM^{\RR_+ \times \Sigma}(\gamma_+, \gamma_-; H,J)$ with $\lim_{j\to \infty} \inf(\pi_{\RR_+}(u_j)) = 0$. By the usual SFT-compactness, and since $H$ is constant on the negative end, they converge to a broken cylinder (see Figure~\ref{fig:cylinder_breaking}). Its top level component is a Floer cylinder with punctures, at which it is asymptotic to contractible Reeb orbits $c_1, \ldots, c_k$. As was shown in \cite[Section~5.2]{BO_seq}, the domain of the top component is connected. (The reason is that the $\RR_+$-component of the Floer cylinder approaches the orbit $\gamma_-$ from above, hence a breaking as in Figure~\ref{fig:cylinder_not_possible} is prevented by the maximum principle.) The moduli space of such punctured Floer cylinders has virtual dimension
\begin{equation} \label{eq_dim_top_comp}
 \mu(\gamma_+)-\mu(\gamma_-) - \sum_{j=1}^k (\mu(c_j)+n-3) -1 = - \sum_{j=1}^k (\mu(c_j)+n-3),
\end{equation}
where the $-1$ comes from dividing out the free $\RR$-action by shifts in the domain (see \cite[Section~5.2]{BO_seq}).
By the assumption on the indices of contractible Reeb orbits, this dimension is negative. Hence, by transversality (assuming $J_t$ was chosen sufficiently generic), this space is empty, giving a contradiction.
\end{proof}

In the same way, one can show that the moduli spaces for continuation maps are compact. In this case, there is no $\RR$-action divided out, so the virtual dimension is bigger by one compared to \eqref{eq_dim_top_comp}. However, the difference of Conley--Zehnder indices $\mu(\gamma_+) - \mu(\gamma_-)$ is zero, hence one gets the same contradiction.

\begin{cor} \label{cor_defined_in_sympl}
 Assume that $c_1(\Sigma)=0$ and either
 \begin{enumerate}[(i)]
  \item $\mu_{CZ}(c) > 4-n$ for all contractible Reeb orbits $c$, or
  \item $\Sigma$ admits a Liouville filling $W$ with $c_1(W)=0$ and $\mu_{CZ}(c) > 3-n$ for all Reeb orbits $c$ which are contractible in $W$.\footnote{For a Reeb orbit $c$ that is contractible in $W$ but not in $\Sigma$, we have to use the grading $\mu_{CZ}(c)$ coming from a filling disk in $W$.}
 \end{enumerate}
 Then, $\check{SH}$ can be defined by counting Floer cylinders on the symplectization $\RR_+ \times \Sigma$ instead of a filling.
 \end{cor}
 
\begin{proof}
 In addition to the compactness of the moduli spaces for the differential and the continuation maps, we have to show that $\dd \circ \dd = 0$. As usual, this is done by examining the moduli spaces $\mcM^{\RR_+ \times \Sigma}(\gamma_+, \gamma_-; H,J)$ for $\mu(\gamma_+) - \mu(\gamma_-) = 2$. We have to prove again that its elements do not escape to the negative end of the symplectization, so that the moduli space has the usual compactification by products of one-dimensional moduli spaces.
 
 If $\mu_{CZ}(c) > 4-n$ for all contractible Reeb orbits $c$, we can use the same proof as for Lemma~\ref{lem_cyl_cpt}. Indeed, the virtual dimension of the top component is
 \[
  \mu(\gamma_+)-\mu(\gamma_-) - \sum_{j=1}^k (\mu(c_j)+n-3) -1 = 1 - \sum_{j=1}^k (\mu(c_j)+n-3),
 \]
 which is again negative by the stronger index assumption.
 
 If, on the other hand, we only know $\mu_{CZ}(c) > 3-n$, this strategy does not work, since the virtual dimension might just be zero. Instead, if $(ii)$ holds, the strategy is to show that the differential defined by Lemma~\ref{lem_cyl_cpt} and the differential defined by the filling coincide. We have to show that, for any orbits $\gamma_+, \gamma_-$ with $\mu(\gamma_+) - \mu(\gamma_-) = 1$, the moduli spaces
 \begin{equation} \label{eq_moduli_space_bijection}
  \mcM^{\RR_+ \times \Sigma}(\gamma_+, \gamma_-; H,J) \quad \text{ and } \quad \mcM^{\widehat W}(\gamma_+, \gamma_-; H,J)
 \end{equation}
are in bijective correspondence. We use the ``neck-stretching'' operation as in \cite[Section~5.2]{BO_seq}. This basically means that we insert a piece of the symplectization with constant Hamiltonian near $\dd W \cong \{1\} \times \Sigma \subset \widehat W$ and make this piece larger and larger. Under this operation, the elements of $\mcM^{\widehat W}(\gamma_+, \gamma_-; H,J)$ which are not contained in $\RR_{\geq 1}\times \Sigma \subset \widehat W$ converge to broken cylinders as in the right of Figure~\ref{fig:cylinder_breaking}. However, by the same index calculation as in Lemma~\ref{lem_cyl_cpt}, such a breaking is not possible. Hence, this neck-stretching operation gives the correspondence \eqref{eq_moduli_space_bijection}.
\end{proof}

\begin{rmk} \label{rmk_grading}
 In case $(ii)$ of Corollary~\ref{cor_defined_in_sympl}, one can wonder whether $\check{SH}$ is independent of the choice of filling $W$. Indeed, the only place where the choice of $W$ still plays a role is the grading. For a Reeb orbit $c$ which is not contractible in $\Sigma$, the grading generally depends on the choice of a ``reference loop'' in the free homotopy class of $c$. If $c$ is contractible in $W$, however, $W$ gives a canonical choice of grading. This grading might differ for different Liouville fillings with $c_1(W)=0$.
 
 Apart from this grading ambiguity, $\check{SH}$ is independent of $W$. In particular, this is the case if $\pi_1(\Sigma)=0$, or more generally if the induced map $\pi_1(\Sigma) \to \pi_1(W)$ is injective.
 
 Once product structures are taken into account, the grading issue becomes more complicated. Then, the reference loops for different free homotopy classes can no longer be chosen independently from each other, and it is not clear what choices one has in general for the grading of non-contractible orbits. One possible way to go is to split symplectic homology into different homology classes in $H_1(W)$, as opposed to free homotopy classes. If $H_1(W)$ is free, one can assign gradings consistently as in \cite{EGH}. However, if $H_1(W)$ has torsion, one runs into the same problems as in \cite[Section~2.9.1]{EGH}.
 
 To avoid these issues, we assume from now on that $\pi_1(\Sigma)=0$. The only exception in this text will be the example of $A_k$-surface singularities in Section~\ref{sec_A_k_surface}, but these have an explicit Liouville filling with $c_1(W)=0$ and $\pi_1(W)=0$ which can be used to define the grading. 
 
 Alternatively, one can consider the subring $\check{SH}^{\mathrm{contractible}} \subset \check{SH}$ generated by contractible Reeb orbits, for which the grading is always well-defined.
\end{rmk}

\begin{rmk} \label{rmk_SH_plus_indep}
 The statements of Corollary~\ref{cor_defined_in_sympl} and Remark~\ref{rmk_grading} hold equally true for $SH^+$ instead of $\check{SH}$. In particular, if $\Sigma$ is simply-connected and fulfills $c_1(\Sigma)=0$, $\mu_{CZ}(c)>3-n$ for all Reeb orbits $c$ and admits a Liouville filling $W$ with $c_1(W)=0$, then $SH^+(W)$ is independent of the choice of $W$.
\end{rmk}

\begin{defn}
 We call a contact manifold $(\Sigma, \xi)$ \emph{index-positive} if there exists a contact form $\alpha$ with $\xi = \ker(\alpha)$ such that the assumption of Corollary~\ref{cor_defined_in_sympl} is satisfied.
\end{defn}

In the following, we will always assume that $\Sigma$ is index-positive. 
In view of Corollary~\ref{cor_defined_in_sympl}, we will also write $\check{SH}(\Sigma)$ instead of $\check{SH}(W)$.

We would like to have statements analogous to Lemma~\ref{lem_cyl_cpt} and Corollary~\ref{cor_defined_in_sympl} also for moduli spaces of pairs-of-pants. However, there is an additional complication: While the top component of a broken Floer cylinder was always connected, a pair-of-pants can also break as in Figure~\ref{fig:pants_breaking}. We must exclude this by another index condition.

\begin{figure}[ht]
 \centering
 \vspace{0.3cm}
 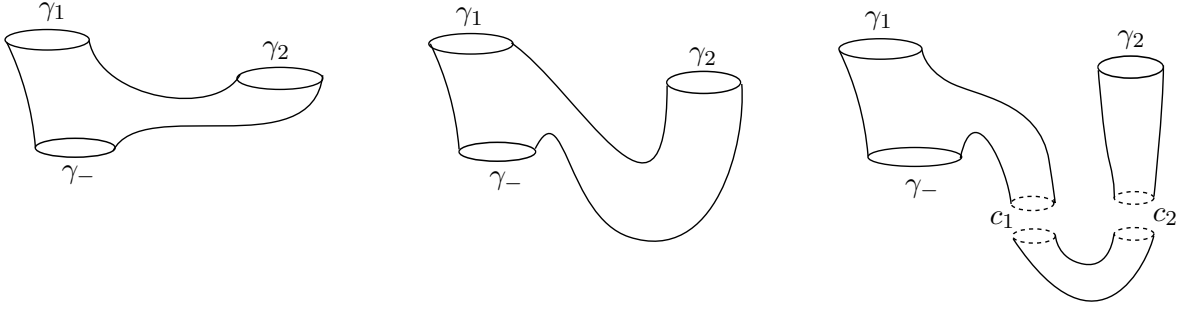
 \caption{Possible breaking of pairs-of-pants. Hamiltonian orbits are represented by continuous lines, Reeb orbits by dashed lines.}
 \label{fig:pants_breaking}
\end{figure}

As as preparation, the next lemma gives the general dimension formula for the moduli spaces of broken Floer curves that appear in the limit process. As always in this section, we assume that $c_1(\Sigma)=0$.

Let $\Gamma^+ = (\gamma_1^+, \ldots, \gamma_{k_+}^+)$ and $\Gamma^- = (\gamma_1^-, \ldots, \gamma_{k_-}^-)$ be collections of Hamiltonian orbits in $\RR_+ \times \Sigma$ and $C = (c_1, \ldots, c_\ell)$ be a collection of contractible Reeb orbits of $\Sigma$. Further, let $H,J, \beta$ be Floer data as in Section~\ref{sec_product_def} (with the straightforward generalization to any number of positive and negative punctures). Denote by $\mcM(\Gamma^+, \Gamma^-, C; \beta, H, J)$ the moduli space of maps 
\[
  u\colon \CP^1\setminus\{z_1^+, \ldots, z_{k_+}^+, z_1^-, \ldots, z_{k_-}^-, \tilde z_1, \ldots, \tilde z_\ell\} \longrightarrow \RR_+\times \Sigma
 \]
which fulfill Floer's equation \eqref{eq_Floer_beta}, converge to $\gamma_i^\pm$ as $z\to z_i^\pm$ in the sense of Floer theory and converge to $\{0\}\times c_j$ at $\tilde z_j$ in the sense of SFT. The conformal structure on $\CP^1 \setminus \{z_1^+, \ldots, z_{k_+}^+, z_1^-, \ldots, z_{k_-}^-\}$ is understood to be fixed, while the points $\tilde z_1, \ldots, \tilde z_\ell$ can vary freely. 

\begin{lem} \label{lem_dim_broken_curves}
 The virtual dimension of this moduli space is
 \[
  \dim \mcM(\Gamma^+, \Gamma^-, C; \beta, H, J) = \sum_{i=1}^{k_+} \mu(\gamma_i^+) - \sum_{i=1}^{k_-} \mu(\gamma_i^-) + n(2-|\Gamma^+|-|\Gamma^-|) - \sum_{j=1}^\ell (\mu(c_j) + n-3).
 \]
\end{lem}

\begin{proof}
 For $C=\emptyset$, the formula is fairly standard (see e.g.\ \cite[Theorem~3.3.11]{Schwarz}). The general case can be deduced by gluing $J$-holomorphic discs to the orbits $c_j$. By \cite[Section~3]{BO_seq}, the dimension of the moduli space of $J$-holomorphic discs asymptotic to a Reeb orbit $c_j$ is 
 \[
  \mu(c_j)+n-3.
 \]
 As the dimension formula is additive under gluing, the result follows.
\end{proof}

\begin{rmk}
 For $|\Gamma^+|=|\Gamma^-|=1$, this is the moduli space of punctured holomorphic cylinders. For this case, the dimension was already computed in \cite{BO_seq}, and we applied the result in the proof of Lemma~\ref{lem_cyl_cpt} above. In the following lemma, we need the cases  $|\Gamma^+|=|\Gamma^-|=1$ and $|\Gamma^+|=1$, $|\Gamma^-|=0$, as these cases appear in Figure~\ref{fig:pants_breaking}.
\end{rmk}

\begin{lem} \label{lem_pop_cpt}
 Fix Hamiltonian orbits $\gamma_1, \gamma_2, \gamma_-$ with  
 \begin{equation} \label{eq_mod_prod_dim_zero}
  \mu(\gamma_1) + \mu(\gamma_2) - \mu(\gamma_-) -n = 0.
 \end{equation}
 Assume that $\Sigma$, in addition to being index-positive, satisfies
 \begin{equation} \label{eq_assumption_prod}
  \mu(c) > \max\{3-|\mu(\gamma_1)|, 3-|\mu(\gamma_2)|\}
 \end{equation}
 for all Reeb orbits $c$. Then, the moduli space $\mcM^{\RR_+ \times \Sigma}(\gamma_1, \gamma_2, \gamma_-; \beta, H,J)$ is compact.
\end{lem}

\begin{proof}
 We need to rule out the breaking as in Figure~\ref{fig:pants_breaking} (and similarly with $\gamma_1$ and $\gamma_2$ exchanged). Then, the rest of the proof works as in Lemma~\ref{lem_cyl_cpt}.

 For the top level on the right of Figure~\ref{fig:pants_breaking} to have positive dimension, by Lemma~\ref{lem_dim_broken_curves}, we would need
 \[
  \mu(\gamma_1) - \mu(\gamma_-) - \mu(c_1) - n +3 \geq 0 
 \]
 and 
 \[
  \mu(\gamma_2) - \mu(c_2) +3 \geq 0.
 \]
 Using \eqref{eq_mod_prod_dim_zero}, these conditions simplify to 
 \[
  \mu(c_1) \leq 3 -\mu(\gamma_2) \quad \text{ and } \quad \mu(c_2) \leq 3 +\mu(\gamma_2).
 \]
 By the assumption \eqref{eq_assumption_prod}, these two equations lead to
 \[
   3-|\mu(\gamma_2)| < 3 -\mu(\gamma_2) \quad \text{ and } \quad 3-|\mu(\gamma_2)| < 3 + \mu(\gamma_2).
 \]
 The first equation implies $\mu(\gamma_2)<0$ while the second equation implies $\mu(\gamma_2)>0$, giving a contradiction.
\end{proof}

\begin{rmk}
 The cylinder in the bottom level on the right of Figure~\ref{fig:pants_breaking} is a holomorphic curve of the kind studied in SFT. As such, it lives in a moduli space of virtual dimension $\mu(c_1)+\mu(c_2)$ (which might not be cut out transversally). Thus, it seems that the virtual dimensions appearing in Figure~\ref{fig:pants_breaking} are not additive under gluing. The reason for the mismatch is that upon gluing, one does in general not recover the conformal structure that was fixed in the left part of Figure~\ref{fig:pants_breaking}. 
\end{rmk}

To have the chain level product of any orbits well-defined in the symplectization, Lemma~\ref{lem_pop_cpt} implies that the condition
\begin{equation} \label{eq_assumption_prod_all}
 \mu(c) > 3 \quad \text{ for all closed Reeb orbits $c$}
\end{equation}
is sufficient. In order for the product to descend to homology, one also need compactness of the one-dimensional moduli spaces. However, a quick calculation (as in the proof of Lemma~\ref{lem_pop_cpt}) shows that \eqref{eq_assumption_prod_all} is sufficient for this as well.

\begin{defn}
 We call a contact manifold $(\Sigma, \xi)$ with $c_1(\Sigma)=0$ and $\pi_1(\Sigma)=0$ \emph{product-index-positive} if there exist a contact form $\alpha$ with $\xi = \ker(\alpha)$ such that \eqref{eq_assumption_prod_all} holds.
\end{defn}

As $\dim(\Sigma)=2n-1$, we have $n\geq 1$, so product-index-positivity implies index-positivity.

\begin{cor} \label{cor_prod_in_sympl}
 For a product-index-positive contact manifold $\Sigma$, $\check{SH}$ and its product structure can be defined by counting Floer cylinders and pairs-of-pants in the symplectization $\RR_+ \times \Sigma$.
\end{cor}

\subsection{$S^1$-actions by Hamiltonian loops on $\RR_+ \times \Sigma$} \label{sec_action_symplectization}

Let $\Sigma$ be any contact manifold for which the Reeb flow is periodic. 
After normalizing the period to one, the Reeb flow defines an $S^1$-action, which we denote by $e^{2\pi i t}.z$, with $t\in S^1=\RR/\ZZ$. Using this, we can define a loop of Hamiltonian diffeomorphisms
\begin{equation} \label{eq_Ham_loop}
 g_t\colon \RR_+ \times \Sigma \to \RR_+ \times \Sigma, \qquad g_t(r,z) = (r, e^{2\pi i \varphi(t)}. z).
\end{equation}
on the symplectization. Here, $\varphi\colon [0,1] \to \RR$ is any map with $\varphi(0)=0$ and $\varphi(1)\in\ZZ$ (e.g.\ the identity map, though we will also need others below). The corresponding Hamiltonian function $K^g_t$ on $\RR_+ \times \Sigma$ is (up to a possibly time-dependent constant) 
\[
 K^g_t(t,r,z) = \varphi'(t) \cdot r.
\]

The following lemma gives a characterization of the Hamiltonians that can be written as $g_*H$ for $H$ constant and $g$ as in \eqref{eq_Ham_loop}.

\begin{lem} \label{lem_char_H}
 A linear Hamiltonian $G$ on $(\RR_{>0}\times \Sigma, d(r\alpha))$ can be written as $g_*H$ for $H \equiv \text{constant}$ and $g$ as in \eqref{eq_Ham_loop} if and only if its slope $\sigma(t)$ depends only on $t$ and fulfills $\int_0^1 \sigma(t) \, dt \in \ZZ$.
\end{lem}

\begin{proof}
 For a loop of Hamiltonian diffeomorphisms $g_t(x,r)$ and $H \equiv \text{constant}$,
 \[
  g_*H_t = H_t + K_t^g = \text{constant} + \varphi'(t) \cdot r
 \]
 has slope $\sigma(t) = \varphi'(t)$. The integral
 \[
  \int_0^1 \sigma(t,r) \, dt = \int_0^1 \varphi'(t) \, dt
 \]
 is the winding number of the loop $\varphi\colon S^1 \to S^1$, hence it has values in $\ZZ$.

 Conversely, assume the slope $\sigma(t)$ of $G$ fulfills $\int_0^1 \sigma(t) \, dt \in \ZZ$. Then, define
 \[
  \varphi(t) \defeq \int_0^t \sigma(\tau) \, d\tau,
 \]
 which fulfills $\varphi(1) \in \ZZ$ and thus descends to a loop on $S^1$. The corresponding loop of Hamiltonian diffeomorphisms $g_t(r, x) = (r, e^{2\pi i \varphi(t)} \, . \, x)$ is associated with the Hamiltonian $K_t^g = \sigma(t)r$, which coincides (up to a constant) with $G$.
\end{proof}

Note that for $g_*H$, with $g_t$ as in \eqref{eq_Ham_loop}, Lemma~\ref{lem_cyl_cpt} cannot be applied directly, because $g_*H$ is not constant on the negative end. 
However, the bijection of moduli spaces from Lemma~\ref{lem_g_properties} still holds, so the compactness of the moduli space $\mcM(\gamma_+, \gamma_-; H,J)$ induces compactness of the moduli space $\mcM(g\cdot \gamma_+, g\cdot \gamma_-; g_*H, g_*J)$.
This gives a possible definition of $HF_*(g_*H)$, basically as the image of $HF_*(H)$ under $S_g$.

A problem with this definition is that one has to worry about compactness again for the continuation maps. 
We deal with this compactness issue in three steps:

\begin{itemize}
 \item Given a continuation map $\Phi^{H \tilde H}$ between two Hamiltonians $H, \tilde H$ as in Figure~\ref{fig:V_shaped_H}, we get a continuation map between $g_*H$ and $g_*\tilde H$ by using the fact that $g$ gives a bijection of the moduli spaces involved. This means that we can define continuation maps for Hamiltonians within the family $g_*H$ for a fixed $g$.
 \item In Lemma \ref{lem_continuation_defined}, we show that if $g_1$ is homotopic to $g_2$, we can define continuation maps between ${g_1}_*H$ and ${g_2}_*\tilde H$.
 \item In Proposition \ref{prop_cutoff}, we show that we get the same Floer homology as for $g_*H$ if we make the Hamiltonian constant near the negative end of the symplectization. Therefore, this Floer homology can be used in the limit process to $\check{SH}(\Sigma)$.
\end{itemize}

\begin{lem} \label{lem_continuation_defined}
 Let $g_1$ and $g_2$ be homotopic through loops of Hamiltonian diffeomorphisms. Then, for $H, \tilde H$ two Hamiltonians as in Figure~\ref{fig:V_shaped_H} (with $\tilde H$ steeper at~$\infty$ than $H$) and $J, \tilde J$ regular almost complex structures, there exists a continuation map from $({g_1}_*H, {g_1}_*J)$ to $({g_2}_*\tilde H, {g_2}_*\tilde J)$. 
\end{lem}

\begin{proof}
 By concatenation with $g_2^{-1}$, we can reduce the general case to the case $g_2=\id$. 
 Denote by $g_{s,t}$, $s\in\RR$, the homotopy from $g_t$ to $\id$, and arrange it such that $g_{s,t} = \id$ for $s\geq 1$ and $g_{s,t}=g_t$ for $s\leq -1$. By the assumption on the slopes, there is a homotopy $(H_{s,t}, J_{s,t})$ from $(H,J)$ to $(\tilde H, \tilde J)$ that defines a continuation map. In particular, the moduli spaces
 \[
  \mcM(\gamma, \tilde \gamma; H_{s,t}, J_{s,t})
 \]
 are compact for all $H$-orbits $\gamma$ and $\tilde H$-orbits $\tilde \gamma$ with $\mu(\tilde \gamma) -\mu(\gamma) = 0$. Now, we can apply $g_{s,t}$ to its elements. As in Lemma \ref{lem_g_properties} (and because $g_{s,t}=\id$ for $s\geq 1$), this gives a bijective correspondence between the moduli space above and
 \[
  \mcM \left(g_t \cdot \gamma, \tilde \gamma; (g_{s,t})_* H_{s,t}, (g_{s,t})_* J_{s,t} \right).
 \]
 Hence, these moduli spaces are also compact and define a continuation map from $(g_*H, g_*J)$ to $(\tilde H, \tilde J)$.
\end{proof}

\begin{prop} \label{prop_cutoff}
 Denote by $[g_*H]_0$ the Hamiltonian which, up to a smoothing, equals $g_*H$ on $(e^{-T}, \infty)\times \Sigma$ and is constant on $(0, e^{-T})\times \Sigma$. Then, for $T$ is sufficiently large (dependent on $g_*H$), there is a bijection between the zero-dimensional moduli spaces
\[
 \mcM(\gamma_+, \gamma_-;g_*H,g_*J) \cong \mcM(\gamma_+, \gamma_-; [g_*H]_0, g_*J).
\]
\end{prop}

\begin{proof}
Denote by $u_1, \ldots, u_n$ the elements of the moduli space $\mcM(\gamma_+, \gamma_-;g_*H,g_*J)$.
By compactness, they live in a compact region $[e^{-T}, e^T] \times \Sigma$ of the symplectization. We choose this value for $T$. Then, in this region, $g_*H = [g_*H]_0$, hence $u_1, \ldots, u_n$ are also elements of $\mcM(\gamma_+, \gamma_-; [g_*H]_0, g_*J)$.

Assume that the latter moduli space has some further element $u'$. By applying $g_t^{-1}$, this gives an element 
\[
 g^{-1} u' \in \mcM(g^{-1} \gamma_+, g^{-1} \gamma_-; g^*[g_*H]_0, J).
\]
Since $g^*[g_*H]_0 = H = \text{constant}$ on $(e^{-T}, \delta) \times \Sigma$, we can use a neck-stretching operation there, as in the proof of Corollary~\ref{cor_defined_in_sympl}. So we insert a piece of the symplectization near $\{e^{-T}\} \times \Sigma$. Under this operation, the Floer cylinder $g^{-1} u'$ converges to a broken cylinder as in Figure~\ref{fig:cylinder_breaking}. 
However, as in the proof of Lemma~\ref{lem_cyl_cpt}, the index condition on the Reeb orbits makes sure that the cylinder is in fact unbroken. This implies that $g^{-1} u'$ was in fact a Floer cylinder for the original Hamiltonian $H$, hence $u'$ was a Floer cylinder for the Hamiltonian $g_* H$. This contradicts the assumption that $u'$ was not among the elements $u_1, \ldots, u_n$.

Thus, all elements of $\mcM(\gamma_+, \gamma_-; [g_*H]_0, g_*J)$ are already contained in $\mcM(\gamma_+, \gamma_-; g_*H, g_*J)$, which gives the bijection.
\end{proof}

Together, Lemma~\ref{lem_continuation_defined} and Proposition~\ref{prop_cutoff} show that Hamiltonians $g_*H$ (with $H$ as in Figure~\ref{fig:V_shaped_H}) can be used in the definition of $\check{SH}_*(\Sigma)$. Indeed, by Lemma~\ref{lem_continuation_defined}, we can arrange that the slope of $g_*H$ is time-independent. Further, we can use continuation maps from  $g_*H$ to $g_*\tilde H$ such that the slopes $\mu_1, \mu_2$ grow arbitrarily large, while $\delta$ remains small and the slope of $g_* \tilde H$ at the negative end of the symplectization stays constant.
This makes sure any orbits created in the transition from $g_* \tilde H$ to $[g_* \tilde H]_0$ have action outside of the fixed action window $(a,b)$.
Hence, the generators of $HF^{(a,b)}([g_*\tilde H]_0)$ are the same as those of $HF^{(a,b)}(g_* \tilde H)$, and Proposition~\ref{prop_cutoff} shows that the differential agrees as well. As $[g_*\tilde H]_0$ is constant at the negative end, it is clear that it can be used to define $\check{SH}_*(\Sigma)$.

The statements of Lemma~\ref{lem_g_properties} and Corollary~\ref{cor_S_g_iso} hold as in the closed case.

\begin{example} \label{ex_simple_loop}
Take the specific loop of Hamiltonian diffeomorphisms
\begin{equation} \label{eq_simple_loop}
 g_t(r,z) = (r, e^{2\pi i t}.z),
\end{equation}
i.e.\ the case $\varphi = \id_{[0,1]}$, and normalize the corresponding Hamiltonian to
\[
 K_t^g(t,r,z) = r - 1.
\]
Then, for $H_t$ as in Figure~\ref{fig:V_shaped_H} (only dependent on the radial coordinate $r$), the Hamiltonian
\[
 g_*H(t,r,z) = H_t(r) + K_t^g(t,r,z) = H_t(r) + (r-1)
\]
is again normalized such that $g_*H = -\varepsilon$ at $r=1$. Thus, except for the non-zero slope at the negative end (which equals one), $g_*H$ looks as in Figure~\ref{fig:V_shaped_H}, but with $\mu_1$ decreased and $\mu_2$ increased by one, respectively.
As for the action, first note that because of the chain rule
\[
 \frac{d}{dt} (g_t\gamma(t)) = (g_t)_* \gamma'(t) + \frac{d}{d\tau}\Big|_{\tau=t} g_\tau(\gamma(t))
\]
and $g_t^*\alpha = \alpha$, we get that
\[
 \int_{S^1} (g_t \gamma)^* \alpha = \int_{S^1} \gamma^*\alpha + 1.
\]
For the second term,
\[
 -\int_{S^1} (g_* H) (g_t \gamma(t)) \,dt = -\int_{S^1} H(\gamma(t)) \,dt - \underbrace{\int_{S^1} K_t^g(\gamma(t)) \,dt}_{\approx 0},
\]
where the second summand vanishes up to an arbitrary small error due to the smoothing of $H$. Hence, except for this small error,
\begin{equation} \label{eq_action_Sg}
 \mcA_{g_*H}(g\cdot \gamma) = \mcA_H(\gamma) + 1,
\end{equation}
which gives an isomorphism
\[
 S_g \colon HF^{(a,b)}(H_{\mu_1,\mu_2}) \stackrel{\cong}{\longrightarrow} HF^{(a+1,b+1)}(H_{\mu_1-1, \mu_2+1}).
\]
Taking the direct limits $\mu_1, \mu_2 \to \infty$, this induces an isomorphism
\begin{equation} \label{eq_Sg_in_action_window}
S_g \colon \check{SH}^{(a,b)}(\Sigma) \stackrel{\cong}{\longrightarrow} \check{SH}^{(a+1,b+1)}(\Sigma),
\end{equation}
and, after taking the additional limit from \eqref{eq_lim_SH_2}, an isomorphism on $\check{SH}(\Sigma)$, which we still denote by $S_g$.
The multiplication
\begin{equation}\label{eq_module_str}
 \ZZ_2[t,t^{-1}] \times \check{SH}(\Sigma) \to \check{SH}(\Sigma), \qquad (t^k, \gamma) \mapsto S_g^k(\gamma) 
\end{equation}
gives $\check{SH}(\Sigma)$ the structure of a module over the ring $\ZZ_2[t,t^{-1}]$.

We will see shortly that if $I(g)\neq 0$, this module is free and finitely generated. For $I(g)=0$, the situation is somewhat different. We will see that in this case, similar statements hold if one replaces Laurent polynomials by Laurent series.

\begin{lem} \label{lem_module_torsion_free}
 For $I(g)\neq 0$, the module structure of $\check{SH}(\Sigma)$ over $\ZZ_2[t,t^{-1}]$ is torsion-free.
\end{lem}

\begin{proof}
 Fix some $x\in \check{SH}(\Sigma)$. As $I(g)\neq 0$, all the elements $S_g^k(x)$ for $k\in\ZZ$ have different degrees, so they are linearly independent.
\end{proof}

In the following discussion, we will make use of the long exact sequence
\begin{equation} \label{eq_les_SH_filtr}
 \cdots \longrightarrow \check{SH}_k^{(a,b)}(\Sigma) \longrightarrow \check{SH}_k^{(a,c)}(\Sigma) \longrightarrow \check{SH}_k^{(b,c)}(\Sigma) \longrightarrow \check{SH}_{k-1}^{(a,b)}(\Sigma) \longrightarrow \cdots
\end{equation}
for any $-\infty \leq a < b < c \leq \infty$ with $a,b,c\notin \Spec(\Sigma)$. This sequence is induced directly from the short exact sequence of chain complexes
\[
 0 \longrightarrow \check{CF}_k^{(a,b)}(H) \longrightarrow \check{CF}_k^{(a,c)}(H) \longrightarrow \check{CF}_k^{(b,c)}(H) \longrightarrow 0
\]
and the fact that the limits preserve exactness (which is true because the inverse limit is only applied to finite dimensional vector spaces, and is always true for the direct limit, see also \cite[Remark~2.7]{Ciel_Fra_Oan}).
Note that \eqref{eq_les_SH_filtr} implies in particular
\begin{equation} \label{eq_dim_in_filtr_subadditive}
 \dim_{\ZZ_2} \left(\check{SH}_k^{(a,c)}(\Sigma)\right) \leq \dim_{\ZZ_2} \left(\check{SH}_k^{(a,b)}(\Sigma)\right) + \dim_{\ZZ_2} \left(\check{SH}_k^{(b,c)}(\Sigma)\right).
\end{equation}

\begin{lem} \label{lem_SH_fin_gen_Ig_neq_0}
 For $I(g)\neq 0$, the $\ZZ_2[t,t^{-1}]$-module $\check{SH}(\Sigma)$ is finitely generated.
\end{lem}

\begin{proof}
 We first show that if $I(g)\neq 0$, $\dim_{\ZZ_2}(\check{SH}_k(\Sigma)) < \infty$ for any fixed degree $k$. 
 Indeed, by \eqref{eq_dim_in_filtr_subadditive} and the isomorphism
 \[
  S_g \colon \check{SH}_k^{(a,a+1)}(\Sigma) \stackrel{\cong}{\longrightarrow} \check{SH}_{k+2I(g)}^{(a+1,a+2)}(\Sigma),
 \]
 we get\footnote{We use the notation $\check{SH}_k^{\shiftedinterval{n}{m}}(\Sigma)$ with $n,m\in \ZZ$ as shorthand for $\check{SH}_k^{(n-\epsilon, m-\epsilon)}(\Sigma)$ with $\epsilon>0$ sufficiently small.}
 \begin{align}
  \dim_{\ZZ_2}(\check{SH}_k(\Sigma)) & \leq \sum_{\ell \in \ZZ} \dim_{\ZZ_2} \left( \check{SH}^{\shiftedinterval{\ell}{\ell+1}}_k(\Sigma) \right) \nonumber \\
  & = \sum_{\ell\in \ZZ} \dim_{\ZZ_2} \left( \check{SH}^{\shiftedinterval{0}{1}}_{k-2I(g)\cdot \ell}(\Sigma) \right) \label{eq_dimension_estimate} \\
  & \leq \dim_{\ZZ_2} \left( \check{SH}^{\shiftedinterval{0}{1}}(\Sigma) \right), \nonumber
 \end{align}
 which is finite because the action spectrum is discrete.
 
 So now, we know that 
 \begin{equation} \label{eq_fin_dim_for_I_neq_0}
  \sum_{j=0}^{2I(g)-1} \dim_{\ZZ_2}(\check{SH}_j(\Sigma)) < \infty
 \end{equation}
 (and similarly if $I(g)$ is negative). Further, any $x\in \check{SH}_k(\Sigma)$ can be written as
 \[
  x = \left(S_g\right)^i (y),
 \]
 where $0\leq \mu(y) < 2I(g)$ (basically, $y \defeq \left(S_g\right)^{-i} (x)$ with $i = \floor{\frac{\mu(x)}{2I(g)}}$). Hence, $\check{SH}_k(\Sigma)$ is generated by a basis of the vector space in \eqref{eq_fin_dim_for_I_neq_0}.
\end{proof}

It now follows algebraically that if $I(g)\neq 0$, $\check{SH}(\Sigma)$ is a free and finitely generated module over $\ZZ_2[t,t^{-1}]$.
Indeed, as a localization of the principal ideal domain $\ZZ_2[t]$, the ring $\ZZ_2[t,t^{-1}]$ is itself a principal ideal domain (\cite[Exercise~II.4]{Lang}).
It follows from the structure theorem for finitely generated modules over a principal ideal domain (see e.g.\ \cite[Theorem~9.3]{Rotman} for the version we need) that any finitely generated, torsion-free module over a principal ideal domain is free. Hence $\check{SH}(\Sigma)$ is a free and finitely generated $\ZZ_2[t,t^{-1}]$-module. 

The proof of Lemma~\ref{lem_SH_fin_gen_Ig_neq_0} shows that its dimension (i.e.\ the number of generators) is bounded from above by $\sum_{j=0}^{2I(g)-1} \dim_{\ZZ_2}(\check{SH}_j(\Sigma))$. Moreover, the estimate in \eqref{eq_dimension_estimate} can be sharpened to
 \begin{align*}
  \sum_{j=0}^{2I(g)-1} \dim_{\ZZ_2}(\check{SH}_j(\Sigma)) & \leq \sum_{j=0}^{2I(g)-1} \sum_{\ell \in \ZZ} \dim_{\ZZ_2} \left( \check{SH}^{\shiftedinterval{\ell}{\ell+1}}_j(\Sigma) \right) \\
  & = \sum_{j=0}^{2I(g)-1} \sum_{\ell\in \ZZ} \dim_{\ZZ_2} \left( \check{SH}^{\shiftedinterval{0}{1}}_{j-2I(g)\cdot \ell}(\Sigma) \right) \\
  & = \sum_{k\in \ZZ} \dim_{\ZZ_2} \left( \check{SH}^{\shiftedinterval{0}{1}}_k(\Sigma) \right) \\
  & = \dim_{\ZZ_2} \left( \check{SH}^{\shiftedinterval{0}{1}}(\Sigma) \right),
 \end{align*}
so we get $\dim_{\ZZ_2[t,t^{-1}]} \check{SH}(\Sigma) \leq \dim_{\ZZ_2} \check{SH}^{\shiftedinterval{0}{1}}(\Sigma)$.

\bigskip

Now, we turn to the case $I(g)=0$. The key difference here is that all the elements $S_g^k(x)$ for $k\in\ZZ$ have the same degree. So, the analog of Lemma\ \ref{lem_module_torsion_free} does not hold, at least not with the given proof.

However, one should really replace the ring $\ZZ_2[t,t^{-1}]$ with the ring of semi-infinite Laurent series $\ZZ_2((t^{-1}))$. The reason is that because of the inverse limit over $a$ in \eqref{eq_lim_SH_2}, infinite sums of the form
\[
 \sum_{k=-\infty}^N \lambda_k S_g^k(x), \qquad \lambda_k\in\ZZ_2
\]
may appear in $\check{SH}_*(\Sigma)$. Note that this cannot occur if $I(g)\neq 0$, because \eqref{eq_lim_SH_2} fixes the degree before taking the limits.

So, as $\kk \defeq \ZZ_2((t^{-1}))$ is a field, this means that $\check{SH}(\Sigma)$ is a vector space over $\kk$. We will now prove an analog of Lemma~\ref{lem_SH_fin_gen_Ig_neq_0}.

\begin{lem} \label{lem_SH_fin_gen_Ig_eq_0}
 If $I(g) = 0$, $\check{SH}(\Sigma)$ is a finite dimensional vector space over $\ZZ_2((t^{-1}))$.
\end{lem}

Note that this lemma cannot be proven in the same way as Lemma~\ref{lem_SH_fin_gen_Ig_neq_0}, as $\dim_{\ZZ_2}(\check{SH}_k(\Sigma))$ may very well be infinite for some $k$ (and there are examples where it is in fact infinite). 

\begin{proof}
 By the discreteness of the action spectrum, $\check{SH}^{\shiftedinterval{0}{1}}(\Sigma)$ has finite dimension over $\ZZ_2$. Hence, it is non-zero only in finitely many degrees.
 By
 \[
  \check{SH}^{(a,b)}_k(\Sigma) \stackrel{S_g}{\cong} \check{SH}^{(a+1,b+1)}_k(\Sigma)
 \]
 and \eqref{eq_dim_in_filtr_subadditive}, this implies that $\check{SH}_k(\Sigma)$ is non-zero only for finitely many degrees~$k$. Thus, it suffices to show that $\check{SH}_k(\Sigma)$ (which is a sub vector space of $\check{SH}(\Sigma)$) has finite dimension over $\ZZ_2((t^{-1}))$.
 
 We will now examine a general element of $\check{SH}_k(\Sigma)$. First, by the definition of the direct limit over $b$, it can be represented by $x\in \check{SH}^{(-\infty, b)}_k(\Sigma)$ for some $b\in \RR$. Without loss of generality, we can assume $b$ to be an integer.
 
 Furthermore, to resolve the inverse limit over $a$, we use the cofinal inverse subsystem
 \begin{equation} \label{eq_inverse_subsystem}
   \cdots \longrightarrow \check{SH}_k^{\shiftedinterval{-n-1}{b}}(\Sigma) \longrightarrow \check{SH}_k^{\shiftedinterval{-n}{b}}(\Sigma) \longrightarrow \cdots \longrightarrow \check{SH}_k^{\shiftedinterval{-b-1}{b}}(\Sigma)
 \end{equation}
 for $n\in \ZZ$. Thus, the element $x\in \check{SH}^{(-\infty, b)}_k(\Sigma)$ can be written as a sequence
 \begin{equation} \label{eq_general_element_of_SH}
  x = (\ldots, x_{-n-1}, x_{-n}, \ldots, x_{b-1}),
 \end{equation}
 where $x_{-n} \in \check{SH}^{\shiftedinterval{-n}{b}}_k(\Sigma)$ and $x_{-n-1} \longmapsto x_{-n}$ under the maps in \eqref{eq_inverse_subsystem}.

The long exact sequence
\begin{equation} \label{eq_particular_les_SH_filtr}
 \cdots \longrightarrow \check{SH}_k^{\shiftedinterval{-n-1}{-n}}(\Sigma) \stackrel{\alpha}{\longrightarrow} \check{SH}_k^{\shiftedinterval{-n-1}{b}}(\Sigma) \stackrel{\beta}{\longrightarrow} \check{SH}_k^{\shiftedinterval{-n}{b}}(\Sigma) \longrightarrow \cdots
\end{equation}
 from \eqref{eq_les_SH_filtr} can be split into short exact sequences of the form
\[
 0 \longrightarrow \im(\alpha) \longrightarrow \check{SH}_k^{\shiftedinterval{-n-1}{b}}(\Sigma) \longrightarrow \underbrace{\check{SH}_k^{\shiftedinterval{-n-1}{b}}(\Sigma) / \ker(\beta)}_{\cong \im(\beta)} \longrightarrow 0,
\]
 which split because $\ZZ_2$ is a field (although not canonically). So we get the decomposition
 \begin{equation} \label{eq_SH_filtr_decomposition}
  \check{SH}_k^{\shiftedinterval{-n-1}{b}}(\Sigma) \cong  \im(\alpha) \oplus \im(\beta).
 \end{equation}
 The equation $x_{-n-1} \longmapsto x_{-n}$ from above implies that 
 \begin{equation} \label{eq_x__under_decomposition}
  x_{-n-1} = \left(\alpha(y_{-n-1}), x_{-n}\right)
 \end{equation}
 under this decomposition for some $y_{-n-1} \in \check{SH}_k^{\shiftedinterval{-n-1}{-n}}$.
 
 Next, we will construct elements $\tilde v_1, \ldots, \tilde v_m$ that we claim to be a generating set of $\check{SH}_k$ over $\ZZ_2((t^{-1}))$.
 Choose a $\ZZ_2$-basis $v_1, \ldots, v_m$ of the image of the map $\check{SH}_k^{(-\infty,0)} \to \check{SH}_k^{\shiftedinterval{-1}{0}}$. So in particular, $v_i \in \im(\beta)$ for $\beta$ as in \eqref{eq_particular_les_SH_filtr} with $b=0$.
 Now, define
 \[
  \tilde v_i = (\ldots, \tilde v_{i,-3}, \tilde v_{i,-2}, \tilde v_{i, -1}) \in \check{SH}_k^{(-\infty,0)}(\Sigma), \qquad \tilde v_{i,-n} \in \check{SH}_k^{\shiftedinterval{-n}{0}}(\Sigma)
 \]
 recursively by $\tilde v_{i, -1} = v_i$ and $\tilde v_{i, -n-1} = (0, \tilde v_{i, -n})$ under the decomposition \eqref{eq_SH_filtr_decomposition}. As $v_i$ lies in the image of $\check{SH}_k^{(-\infty,0)} \to \check{SH}_k^{\shiftedinterval{-1}{0}}$, we know that $\tilde v_{i, -n} \in \im(\beta)$ for the corresponding map $\beta$, so this definition works.
 
 To see that $\tilde v_1, \ldots, \tilde v_m$ generate $\check{SH}_k$, consider again the general element of $\check{SH}_k$ from \eqref{eq_general_element_of_SH}. As $x_{b-1}$ necessarily lies in the image of $\check{SH}_k^{(-\infty,b)} \to \check{SH}_k^{\shiftedinterval{b-1}{b}}$, it can be written as a linear combination
 \[
  x_{b-1} = \sum_{i=1}^m \lambda_i S_g^b(v_i)
 \]
 for some $\lambda_i \in \ZZ_2$. Thus, $x$ will have the form
 \[
  x = \sum_{i=1}^m \lambda_i S_g^b(\tilde v_i) + \cdots,
 \]
 where the dots stand for summands where the exponent of $S_g$ is less than $b$.
 
 The remaining summands can be determined inductively, so that more and more sequence elements $x_i$ are correct. For the induction step, suppose that we have a linear combination of $\tilde v_i$ over $\ZZ_2((t^{-1}))$ so that the sequence elements $x_{-n}, \ldots, x_{b-1}$ already match and we want to add terms with $S_g^{-n}$ so that $x_{-n-1}$ matches. By \eqref{eq_x__under_decomposition}, the part in $\im(\beta)$ is fixed by $x_{-n}$ and we have to represent $\alpha(y_{-n-1})\in \im(\alpha)$ by the $\tilde v_i$.

 Consider the commutative diagram
 \[
  \begin{tikzcd}[column sep=large]
   \cdots \arrow[r] & \check{SH}_k^{(-\infty,-n)} \arrow[d] \arrow[r] & \check{SH}_k^{(-\infty,b)} \arrow[d] \arrow[r] & \check{SH}_k^{\shiftedinterval{-n}{b}} \arrow[d, equal] \arrow[r] & \cdots \\
   \cdots \arrow[r] & \check{SH}_k^{\shiftedinterval{-n-1}{-n}} \arrow[r, "\alpha"] & \check{SH}_k^{\shiftedinterval{-n-1}{b}} \arrow[r, "\beta"] & \check{SH}_k^{\shiftedinterval{-n}{b}} \arrow[r] & \cdots
  \end{tikzcd}
 \]
 where the rows are exact sequences. Since $x_{-n-1} = \left(\alpha(y_{-n-1}), x_{-n}\right)$ lies in the image of $\check{SH}_k^{(-\infty,b)} \to \check{SH}_k^{\shiftedinterval{-n-1}{b}}$, so does $\left(\alpha(y_{-n-1}), 0\right)$. The latter element is mapped by $\beta$ to zero, so a quick diagram chase shows that $y_{-n-1}$ lies in the image of $\check{SH}_k^{(-\infty,-n)} \to \check{SH}_k^{\shiftedinterval{-n-1}{-n}}$. Hence
 \[
  y_{-n-1} = \sum_{i=1}^m \mu_i S_g^{-n}(v_i)
 \]
 for some $\mu_i \in \ZZ_2$, and
 \[
  \sum_{i=1}^m \lambda_i S_g^b(\tilde v_i) + \cdots + \sum_{i=1}^m \mu_i S_g^{-n}(\tilde v_i)
 \]
 will match $x$ in the sequence representation \eqref{eq_general_element_of_SH} up to (and including) $x_{-n-1}$. This finishes the induction and hence the proof.
\end{proof}

We sum up this discussion in the following theorem:
\end{example}

\begin{thm} \label{thm_free_module}
 Assume that $\Sigma$ has periodic Reeb flow and is index-positive. If $I(g) \neq 0$ for $g_t$ as in \eqref{eq_simple_loop}, then $\check{SH}(\Sigma)$ is a free and finitely generated module over $\ZZ_2[t, t^{-1}]$, with the module structure from \eqref{eq_module_str}. If $I(g)=0$, then $\check{SH}(\Sigma)$ is a finite-dimensional vector space over $\kk = \ZZ_2((t^{-1}))$. 
 
 In both cases, the dimension is bounded by $\dim_{\ZZ_2} \left(\check{SH}^{\shiftedinterval{0}{1}}(\Sigma)\right)$.
\end{thm}

\begin{rmk}
 If one prefers to work over the ring of Laurent series for the case $I(g)\neq 0$, one can define a variant of $\check{SH}(\Sigma)$, namely 
 \begin{equation*}
  \widetilde{SH}(\Sigma) \defeq \varinjlim_b \varprojlim_a \check{SH}^{(a,b)}(\Sigma).
 \end{equation*}
 The difference with \eqref{eq_lim_SH_2} is that here, we do not fix the grading before taking the limits, so we allow for any infinite sum of terms whose actions go to $-\infty$. Then, similarly to the case $I(g)=0$, $\widetilde{SH}(\Sigma)$ is a finite-dimensional vector space over $\kk = \ZZ_2((t^{-1}))$.
\end{rmk}

\subsection{Homotopy invariance} \label{sec_homot_inv}

This section and the next one are devoted to stating, proving and using the statements of Proposition~\ref{prop_less_basic} in the current setup. $\Sigma$ is assumed to be index-positive.

\begin{prop} \label{prop_homot_invar}
 Let $g_t$ and $\tilde g_t$ be homotopic through a homotopy of loops of Hamiltonian diffeomorphisms $g_{t,r}$ with $g_{0,r} = \id$ for all $r$. Then, 
 the isomorphisms
  \[
   S_g, S_{\tilde g} \colon \check{SH}(\Sigma) \stackrel{\cong}{\longrightarrow} \check{SH}(\Sigma)
  \]
  coincide.
\end{prop}

\begin{proof}
 The proof follows the lines of \cite[Section~5]{Seidel_pi_1} and is a variation of the standard ``homotopy of homotopies'' argument, which is used in Floer homology to show that continuation maps do not depend on the chosen homotopy $(H_s, J_s)$. We omit some of the details that do not differ from the closed case. 
 
 First, note that $S_g$ satisfies the concatenation property
 \[
  S_{g^1} \circ S_{g^2} = S_{g^1 \# g^2}
 \]
 for two loops $g_t^1, g_t^2$ of Hamiltonian diffeomorphisms. Therefore, it suffices to prove the proposition in the special case $g_{t,0} = \tilde g_t \equiv \id$.
 
 Denote by $\tilde H$ a Hamiltonian as in Figure~\ref{fig:V_shaped_H} such that the slopes of $(g_r)_* \tilde H$ at infinity are steeper than those of $H$ for all $r$. Further, let $(H', J')$ be a regular homotopy from $(H,J)$ to $((g_1)_* \tilde H, (g_1)_*J)$ and $(H'', J'')$ a regular homotopy from $(H,J)$ to $(\tilde H,J)$. 

\begin{figure}[ht]
 \centering
 \def\svgwidth{0.35\textwidth}
 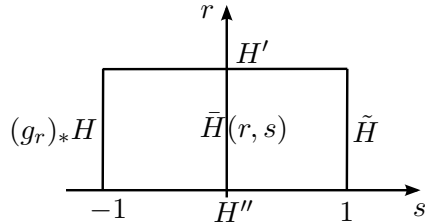
 \caption{Visualization of a deformation of homotopies}
 \label{fig:defor_homot}
\end{figure}
 
 \begin{defn}
  A \emph{deformation of homotopies} is pair of a function $\bar H\in C^\infty([0,1]\times\RR\times S^1 \times \Sigma, \RR)$ and a family of $\omega$-compatible almost complex structures $(\bar J_{r,s,t})$ parametrized by $(r,s,t)\in [0,1] \times \RR \times S^1$ such that
  \begin{align*}
   & \bar H(r,s,t,x) = H(t,x), \quad & & \bar J_{r,s,t} = J_t & & \text{for} \quad s\leq -1, \\
   & \bar H(r,s,t,x) = \left((g_r)_* \tilde H \right) (t,x), \quad & & \bar J_{r,s,t} = (g_r)_*J_t & & \text{for} \quad s\geq 1, \\
   & \bar H(0,s,t,x) = H''(s,t,x), \quad & & \bar J_{0,s,t} = J''_t & & \text{and} \\
   & \bar H(1,s,t,x) = H'(s,t,x), \quad & & \bar J_{1,s,t} = J'_t. & &
  \end{align*}
 \end{defn}

See Figure~\ref{fig:defor_homot} for a visualization. 
By Lemma~\ref{lem_char_H}, we can choose a deformation of homotopies $(\bar H, \bar J)$ such that on the negative end of the symplectization, $\bar H$ is of the form $g_*H$ for some $g$ as in \eqref{eq_Ham_loop} and $H$ constant. This makes sure that Floer cylinders for $\bar H$ do not escape to the negative end of the symplectization, as in Lemma~\ref{lem_cyl_cpt}.

For $\gamma_-$ an $H$-orbit and $\gamma_+$ an $\tilde H$-orbit, define the moduli space 
\[
 \mcM^h(\gamma_+,\gamma_-; \bar H, \bar J)
\]
as the set of pairs $(r,u)\in [0,1]\times C^\infty(\RR\times S^1, \RR_+ \times \Sigma)$ satisfying
\begin{equation} \label{eq_Floer_barH}
 \dd_s u + \bar J_{r,s,t}(u(s,t))\left(\dd_t u - X_{\bar H}(r,s,t,u(s,t))\right) = 0
\end{equation}
and the asymptotic conditions
\[
 \lim_{s\to -\infty} u(s) = \gamma_-, \quad \lim_{s\to \infty} u(s) = g_r(\gamma_+).
\]
For a sufficiently generic choice of $(\bar H, \bar J)$, this is a smooth manifold of dimension
\[
 \dim \mcM^h(\gamma_+,\gamma_-; \bar H, \bar J) = \mu(\gamma_+) - \mu(\gamma_-) +1.
\]
Its boundary consists of solutions of \eqref{eq_Floer_barH} with $r=0$ or $r=1$. In these cases, equation \eqref{eq_Floer_barH} becomes
\begin{equation}\label{eq_Floer_r0}
  \dd_s u + \bar J''_{s,t}(u(s,t))\left(\dd_t u - X_{H''}(s,t,u(s,t))\right) = 0
\end{equation}
and 
\begin{equation}\label{eq_Floer_r1}
  \dd_s u + \bar J'_{s,t}(u(s,t))\left(\dd_t u - X_{H'}(s,t,u(s,t))\right) = 0,
\end{equation}
respectively. These are precisely the equations for the continuation maps corresponding to $(H'',J'')$ and $(H', J')$ respectively. 

\begin{lem} 
 \begin{enumerate}[(i)]
  \item If $\mu(\gamma_+) = \mu(\gamma_-) -1$, the moduli space $\mcM^h(\gamma_+,\gamma_-; \bar H, \bar J)$ is a finite set.
  \item If $\mu(\gamma_+) = \mu(\gamma_-) = k$, $\dim \mcM^h(\gamma_+,\gamma_-; \bar H, \bar J) = 1$, and there is a smooth compactification $\overline \mcM^h(\gamma_+,\gamma_-; \bar H, \bar J)$ whose boundary consists, in addition to $\dd \mcM^h(\gamma_+,\gamma_-; \bar H, \bar J)$, of elements of
  \begin{equation} \label{eq_ov_Mh_1}
   \mcM^h(\gamma_+,\gamma; \bar H, \bar J) \times (\mcM(\gamma, \gamma_-; \tilde H, J) / \RR)
  \end{equation}
 for $\gamma$ an $\tilde H$-orbit of index $\mu(\gamma)=k+1$ and
  \begin{equation} \label{eq_ov_Mh_2}
   (\mcM(\gamma_+,\gamma'; H, J)/\RR) \times \mcM^h(\gamma', \gamma_-; \bar H, \bar J)
  \end{equation}
 for $\gamma'$ an $H$-orbit of index $\mu(\gamma)=k-1$.
 \end{enumerate}
 \label{lem_Mh_compact}
\end{lem}

See \cite{Seidel_pi_1} and its references for the proof of Lemma~\ref{lem_Mh_compact}. By this compactness result, it makes sense to define a map
\begin{align*}
 h_k^{\bar H, \bar J} \colon CF_k(H) & \longrightarrow CF_{k+1}(\tilde H) \\
 \gamma_+ & \longmapsto \sum_{\substack{\gamma_- \\ \mu(\gamma_-)=k+1}} \# \mcM^h(\gamma_+, \gamma_-; \bar H, \bar J) \, \gamma_-.
\end{align*}

\begin{lem} \label{lem_chain_homot}
 For all $k$, 
 \begin{equation*}
  \dd_{k+1}^{\tilde H, J} \circ h_k^{\bar H, \bar J} + h_{k-1}^{\bar H, \bar J} \circ \dd_k^{H, J} = \Phi_k^{H', J'} \circ S_{g_{t,1}} - \Phi_k^{H'', J''},
 \end{equation*}
 where $\Phi_k$ denotes the continuation map.
\end{lem}
\begin{proof}
By definition,
\begin{align*}
 \Phi_k^{H', J'} \circ S_{g_{t,1}}(\gamma_+) & = \sum_{\gamma_-} \# \left(\mcM^\Phi(g_1(\gamma_+),\gamma_-; H', J') \right) \gamma_- \\
 & \stackrel{\eqref{eq_Floer_r1}}{=} \sum_{\gamma_-} \# \left\{ (1,u)\in \mcM^h(\gamma_+,\gamma_-; \bar H, \bar J) \right\} \gamma_-.
\end{align*}
As \eqref{eq_Floer_r0} is the Floer equation for the continuation map $\Phi^{H'', J''}$, this implies
\begin{equation*}
 \left(\Phi_k^{H', J'} \circ S_{g_{t,1}} - \Phi_k^{H'', J''} \right)(\gamma_+) = \sum_{\gamma_-} \# \left(\dd \mcM^h(\gamma_+,\gamma_-; \bar H, \bar J)\right) \gamma_-.
\end{equation*}
Moreover, since $\overline \mcM^h(\gamma_+,\gamma_-; \bar H, \bar J)$ is a compact $1$-dimensional manifold with boundary, its boundary has an even number of points. Hence, for $\ZZ_2$-coefficients, we can replace $\# \left(\dd \mcM^h(\gamma_+,\gamma_-; \bar H, \bar J)\right)$ with the contributions from \eqref{eq_ov_Mh_1} and \eqref{eq_ov_Mh_2}. These equations count contributions from the composition of $h_k$ with the differential, thus giving
\begin{equation*}
 \left(\Phi_k^{H', J'} \circ S_{g_{t,1}} - \Phi_k^{H'', J''} \right)(\gamma_+) = \left(\dd_{k+1}^{\tilde H, J} \circ h_k^{\bar H, \bar J} + h_{k-1}^{\bar H, \bar J} \circ \dd_k^{H, J} \right) (\gamma_+),
\end{equation*}
which proves the lemma.
\end{proof}

The statement of Lemma~\ref{lem_chain_homot} means that $\Phi_k^{H', J'}\circ S_{g_{t,1}}$ is chain homotopic to a continuation map. Thus, up to continuation maps, $S_{g_{t,1}}$ is the identity map on Floer homology. 
\end{proof}

\subsection{Application to product computations} \label{sec_application_product}

\begin{prop} \label{prop_relation_product}
Assume that $\Sigma$ is product-index-positive.
The isomorphism $S_g \colon \check{SH}_*(\Sigma) \to \check{SH}_{*+2I(g)}(\Sigma)$ satisfies the relation
 \begin{equation} \label{eq_relation_product}
  S_g(x\cdot y) = S_g(x) \cdot y
 \end{equation}
with the product on $\check{SH}(\Sigma)$.
\end{prop}

\begin{proof}
 Having established Proposition~\ref{prop_homot_invar}, the proof is essentially the same as in \cite[Theorem~23]{Ritter_neg_line} and \cite[Proposition~6.3]{Seidel_pi_1}. 
 Namely, by Proposition~\ref{prop_homot_invar}, we can homotope $g_t$ to another loop of Hamiltonian diffeomorphisms satisfying $g_t = \id$ for $t\in (-\epsilon, \epsilon)$ for some $0<\epsilon<1/4$. 
 
 For the domain of the pair-of-pants, we take the specific surface $\RR\times S^1\setminus \{(0,0)\}$. Choose a cylindrical parametrization $(s,t)$ near $\{0,0\}$, e.g.\ 
 \[
  e(s,t) = \left(\frac{1}{4}e^{-2\pi s}\cos (2\pi t), \frac{1}{4}e^{-2\pi s}\sin (2\pi t) \right)
 \]
 with $s\in (-\infty, 0)$. Let $\gamma_+, \gamma_0, \gamma_-$ be $1$-periodic orbits of $H_+, H_0, H_-$, respectively, and choose $\beta, H_\mcP$ and $J_\mcP$ as in Section~\ref{sec_product_def}. 
Then, the product counts maps
 \[
  u \colon \RR\times S^1 \to \RR_+ \times \Sigma
 \]
satisfying
\[
 (du - X_H\otimes \beta)^{0,1} = 0,
\]
with the asymptotic conditions
\[
 \lim_{s\to \pm\infty} u(s,t) = \gamma_\pm(t),
\]
at the punctures $\pm \infty$ and
\[
 \lim_{s\to \infty} u\circ e(s,t) = \gamma_0(t)
\]
at the puncture $(0,0)$. Since $g_t=\id$ in a neighborhood of $t=0$, we note that $g\cdot u$ satisfies the asymptotic conditions
\[
 \lim_{s\to \pm\infty} (g\cdot u)(s,t) = (g\cdot \gamma_\pm)(t) \quad \text{and} \quad \lim_{s\to \infty} (g\cdot u)\circ e(s,t) = \gamma_0(t).
\]
Hence, the assignment $u\mapsto g\cdot u$ gives a bijection of moduli spaces
\[
 \mcM(\gamma_+, \gamma_0, \gamma_-; \beta, H_\mcP, J_\mcP) \cong \mcM(g\cdot \gamma_+, \gamma_0, g\cdot \gamma_-; \beta, g_*H_\mcP, g_*J_\mcP).
\]
By an analog of Proposition~\ref{prop_cutoff} for pairs-of-pants (which holds by the same proof), the moduli space on the right-hand side does not change if we cut off $g_*H_\mcP$ to a constant near the negative end of the symplectization. Therefore, the elements of the right-hand side are counted by the product 
\[
HF_*(g_*H_+)\times HF_*(H_0) \to HF_*(g_*H_-)
\]
of the elements $S_g(\gamma_+) = g\cdot \gamma_+$, $\gamma_0$ and $S_g(\gamma_-) = g\cdot \gamma_-$,
while the elements of the left-hand side are counted by the product
\[
HF_*(H_+)\times HF_*(H_0) \to HF_*(H_-)
\]
of the elements $\gamma_+$, $\gamma_0$ and $\gamma_-$.
Hence, when taking direct limits to pass to $\check{SH_*}(\RR_+ \times \Sigma)$, this bijection of moduli spaces gives 
\[
 \langle S_g(\gamma_+) \cdot \gamma_0, \: S_g(\gamma_-) \rangle = \langle \gamma_+ \cdot \gamma_0, \: \gamma_- \rangle.
\]
Since the right-hand side is the same as $\langle S_g(\gamma_+ \cdot \gamma_0), S_g(\gamma_-) \rangle$, this implies \eqref{eq_relation_product}.
\end{proof}

As mentioned in Section~\ref{sec_V_shaped}, the ring structure on $\check{SH}(\Sigma)$ has a unit, coming from the generator of $H^0(\Sigma)$. Hence, we can use \eqref{eq_relation_product} with $x=1$ being the unit and $y=\gamma$ some other generator, getting
\begin{equation} \label{eq_Sg_prod}
 S_g(\gamma) = S_g(1\cdot \gamma) = S_g(1)\cdot \gamma.
\end{equation}
Specifically, choose $g_t$ to be the simple loop \eqref{eq_simple_loop} from Example~\ref{ex_simple_loop}. For this case, define 
\[
 s \defeq S_g(1) \in \check{SH}_{n+2I(g)}(\Sigma). 
\]
Hence,
\begin{equation} \label{eq_Sg_prod_s}
 S_g(\gamma) = s \cdot \gamma,
\end{equation}
and similarly $S_g^{-1}(\gamma) = s^{-1}\cdot \gamma$, where $s^{-1}$ is the inverse of $s$ in the ring $\check{SH}(\Sigma)$.

\begin{cor} \label{cor_S_g_is_mult}
 The isomorphism $S_g$ is simply (left-) multiplication by the element $s \in \check{SH}(\Sigma)$. 
 In particular, the structure of $\check{SH}(\Sigma)$ as a module over the ring of Laurent polynomials from \eqref{eq_module_str} is given by\footnote{We renamed the variable of the Laurent polynomials from $t$ to $s$ to emphasize that $s$ is itself an element and $\ZZ_2[s,s^{-1}]$ is a subset of $\check{SH}(\Sigma)$.}
 \begin{equation*}
 \ZZ_2[s,s^{-1}] \times \check{SH}(\Sigma) \to \check{SH}(\Sigma), \qquad (s^k, \gamma) \mapsto s^k \cdot \gamma.
 \end{equation*}

\end{cor}

While the proof given above, specifically Proposition~\ref{prop_relation_product}, was given under the assumption that $\Sigma$ is product-index-positive, it turns out that, at least if $I(g)\neq 0$, a weaker assumption suffices:

\begin{prop} \label{prop_weaker_assumption}
Assume that $\Sigma$ is simply-connected,\footnote{Again, the assumption $\pi_1(\Sigma)=0$ is used only to have a grading of $\check{SH}$ compatible with the product structure and the broken curve in Figure~\ref{fig:pants_breaking}, see Remark~\ref{rmk_grading}.} admits a Liouville filling $W$ with $c_1(W)=0$ and fulfills $\mu_{CZ}(c) > 3-n$ for all Reeb orbits $c$ (i.e.\ it fulfills condition (i) in the definition of index-positivity). Assume further that $I(g)\neq 0$ (for $g$ as in \eqref{eq_simple_loop}). Then, although one needs the filling $W$ to the define the product structure, equation \eqref{eq_Sg_prod_s} (and hence Corollary~\ref{cor_S_g_is_mult}) holds as before.
\end{prop}

\begin{proof}
 As $\Sigma$ is index-positive, both $S_g$ and $s\defeq S_g(1)$ are still well-defined. By Lemma~\ref{lem_pop_cpt}, the product $\gamma_1\cdot \gamma_2$ can be computed in the symplectization if \eqref{eq_assumption_prod} holds. As $\mu(c)>3-n$ for all Reeb orbits $c$, this is guaranteed if $|\mu(\gamma_1)|\geq n$ and $|\mu(\gamma_2)|\geq n$.
 
 Therefore, the proof of \eqref{eq_relation_product} goes through as before if 
 \[
  |\mu(x)|\geq n, \quad |\mu(y)|\geq n \quad \text{ and } \quad |\mu(S_g(x))|\geq n.
 \]
 Recall that the unit has degree $n$, so we can use it for $x$ or $y$. Without loss of generality, assume that $I(g)>0$ (otherwise replace $g$ by its inverse). Then, $\mu(s^k)\geq n$ for all $k\geq 0$, so we can use \eqref{eq_relation_product} inductively to get
\begin{equation} \label{eq_rel_prod_k_pos}
 S_g^k(1) = s^k \qquad \forall k\geq 0.
\end{equation}
 The next step is to see that $s^N$ is invertible, at least for $N$ sufficiently large. Denote by $g^{-N}$ the $(-N)$-fold cover of $g$ and define $x\defeq S_{g^{-N}}(1)$. For $N$ sufficiently large, $\mu(x) \leq -n$, so we can use \eqref{eq_relation_product} to get
 \[
  x\cdot s^N = S_{g^{-N}}(1) \cdot s^N = S_{g^{-N}}(1 \cdot s^N) = S_g^{-N}(s^N) = 1,
 \]
 where the last step follows from \eqref{eq_rel_prod_k_pos}. Hence, $x = (s^N)^{-1}$. Now, for any generator $\gamma\in \check{SH}(\Sigma)$, choose $N$ sufficiently large so that $\mu((s^N)^{-1} \cdot \gamma)<-n$. Then, we can calculate
 \begin{align*}
  S_g(\gamma) & = S_g \left(s^N\cdot (s^N)^{-1} \cdot \gamma \right) \\
  & = S_g(s^N) \cdot \left( (s^N)^{-1} \cdot \gamma \right) \\
  & \stackrel{\eqref{eq_rel_prod_k_pos}}{=} s^{N+1} \cdot (s^N)^{-1} \cdot \gamma \\
  & = s\cdot \gamma
 \end{align*}
 which finishes the proof. 
\end{proof}

The following theorem summarizes the results of this section:

\begin{thm} \label{thm_summary_RFH}
Assume that $\Sigma$ has periodic Reeb flow and satisfies one on the following:
\begin{itemize}
 \item $\Sigma$ is product-index-positive, or 
 \item $\Sigma$ fulfills $\pi_1(\Sigma)=0$, $\mu_{CZ}(c) > 3-n$ for all Reeb orbits $c$ and admits a Liouville filling $W$ with $c_1(W)=0$.
\end{itemize}
 Let $g_t$ be defined as in \eqref{eq_simple_loop} and assume $I(g)\neq 0$. Then, the multiplication
\begin{equation*}
 \ZZ_2[s,s^{-1}] \times \check{SH}(\Sigma) \to \check{SH}(\Sigma), \qquad (s^k,x) \mapsto s^k\cdot x,
\end{equation*}
where $s=S_g(1)$ and $\cdot$ denotes the pair-of-pants product, gives $\check{SH}(\Sigma)\cong RFH(W)$ the structure of a free and finitely generated module over $\ZZ_2[s,s^{-1}]$. The generators of this module are the unit and possibly other finite linear combinations of Reeb orbits.
In particular, $\check{SH}(\Sigma)$ is finitely generated as an algebra.

If $I(g)=0$ and $\Sigma$ is product-index-positive, the same holds true if we replace $\ZZ_2[s,s^{-1}]$ by $\ZZ_2((s^{-1}))$.
\end{thm}

This theorem also includes the (uninteresting) case when $\check{SH}(\Sigma)=0$, as e.g.\ for the standard contact sphere. Note that by \cite[Theorem~13.3]{Ritter_tqft}, $\check{SH}(\Sigma)\cong RFH(W; \ZZ_2)\neq 0$ is equivalent to $SH(W;\ZZ_2)\neq 0$.

Unfortunately, this theorem does not necessarily give the complete product structure of $\check{SH}(\Sigma)$. Indeed, the module generators might not be algebraically independent (one might be the product of two others), or even the generator $s$ might be the square (or some higher power) of some other generator.

\subsection{Back to usual symplectic homology} \label{sec_back_to_usual_SH}

Finally, we can use Theorem~\ref{thm_summary_RFH} to gain some information about the usual symplectic homology of some Liouville filling $W$ of $\Sigma$ with $c_1(W)=0$. The long exact sequence constructed in \cite{Ciel_Fra_Oan} gives in particular a map
\begin{equation} \label{eq_map_SH_to_RFH}
 f\colon SH_*(W) \to \check{SH}_*(\Sigma).
\end{equation}
This map is constructed as follows: The Floer homology of a Hamiltonian on $W$ as in Figure~\ref{fig:V_shaped_H} with the action window $(-\infty, b)$ is isomorphic to $SH^{(-\infty,b)}(W)$. The Floer homology $HF^{(a,b)}(H)$ used in the definition of $\check{SH}$ arises from dividing out the chains of action less than $a$ (provided that $\mu_1$ is sufficiently large). Thus, the map \eqref{eq_map_SH_to_RFH} is just the quotient map $HF^{(-\infty,b)}(H) \to HF^{(a,b)}(H)$ after taking the appropriate limits.

The next lemma is a special case of \cite[Theorem~10.2(e)]{Ciel_Oan}.

\begin{lem} \label{lem_f_intertwines_product}
 The maps $f$ respects the product structures,
\begin{equation*}
 f(x\cdot y) = f(x) \cdot f(y).
\end{equation*}
\end{lem}

\begin{proof}
 The product on $\check{SH}$ is constructed by applying the limits \eqref{eq_lim_SH_1} and \eqref{eq_lim_SH_2} (in the correct order) to the product
\begin{equation} \label{eq_product_def_on_check_SH}
  HF^{[a, b)}(H) \times HF^{[a, b)}(H) \to HF^{[a+b, 2b)}(2H).
\end{equation}
But \eqref{eq_product_def_on_check_SH} also defines the product on $SH(W)$ of any elements that survive the quotient map $HF^{(-\infty,b)}(H) \to HF^{(a,b)}(H)$.
\end{proof}

One should think of the map $f$ as dividing out a part of the negative symplectic homology $SH_*^-(W)\cong H^{n-*}(W)$. This can be seen most easily from the long exact sequence
\begin{equation} \label{eq_les_SH_RFH}
 \cdots \longrightarrow SH^{-k} \stackrel{h}{\longrightarrow} SH_k \stackrel{f}{\longrightarrow} \check{SH}_k \longrightarrow SH^{-(k-1)} \longrightarrow \cdots
\end{equation}
where the map $h\colon SH^{-k} \to SH_k$ factors by \cite[Proposition~1.3]{Ciel_Fra_Oan} as
\begin{equation} \label{eq_h_factoring}
 SH^{-k}(W) \to H^{-k+n}(W,\dd W) \stackrel{PD}{\longrightarrow} H_{k+n}(W) \stackrel{\incl_*}{\longrightarrow} H_{k+n}(W,\dd W) \to SH_k(W).
\end{equation}
By exactness, the induced map $\bar f\colon SH(W) / \im(h) \to \check{SH}(\Sigma)$ is injective, and $\im(h)$ is a subset of the image of $SH^-(W) \to SH(W)$.

Furthermore, for reasons similar to Lemma~\ref{lem_f_intertwines_product}, $f$ maps the unit of $SH$ to the unit of $\check{SH}$. Indeed, both units have the same definition in terms of orbits of $H$, and it can be checked from \eqref{eq_h_factoring} that $\im(h)$ has no elements of degree $n$. Hence, the generators defining the unit are not divided out by $f$.

\begin{cor}
 $SH(W)/\im(h)$ is a commutative ring with unit.
\end{cor}

\begin{proof}
 As the kernel of the ring homomorphism $f$, $\im(h)\subset SH(W)$ is an ideal, hence the quotient is a ring. 
\end{proof}

\begin{thm}
 For $\Sigma$ and $W$ as in Theorem~\ref{thm_summary_RFH}, $SH(W) / \im(h)$ is a free and finitely generated module over the polynomial ring $\ZZ_2[s]$. In particular, $SH(W)$ is finitely generated as a $\ZZ_2$-algebra.
\end{thm}

\begin{proof}
It follows from the construction of the map $f$ in \cite{Ciel_Fra_Oan} that the image $\im(f)$ consists of all elements of $\check{SH}(\Sigma)$ that are represented by orbits in the regions (IV) and (V) of Figure~\ref{fig:V_shaped_H}. Thus, we can choose generators $b_1, \ldots, b_m$ of $\check{SH}(\Sigma)$ over $\ZZ_2[s,s^{-1}]$ (resp.\ $\ZZ_2((s^{-1}))$ if $I(g)=0$) that are contained in $\im(f)$. (Namely, for $I(g)\neq 0$, any element is represented by a finite sum of orbits, so it suffices to apply $S_g$ to each $b_i$ sufficiently many times. For $I(g)=0$, the same can be done after choosing the generators such that they are represented by finite sums of orbits, e.g.\ choosing the generators as in the proof of Lemma~\ref{lem_SH_fin_gen_Ig_eq_0}.)

Then, apply $S_g^{-1}$ to each $b_i$ until the result $\tilde b_i$ is still in $\im(f)$, but one further application $S_g^{-1}(\tilde b_i)$ is not. Hence, $\tilde b_1, \ldots, \tilde b_m$ generate $\im(f) \cong SH(W) / \im(h)$ as a module over $\ZZ_2[s]$. 
Since $\check{SH}(\Sigma)$ is torsion-free, this module is also torsion-free, hence it is free by \cite[Theorem~9.3]{Rotman} (as $\ZZ_2[s]$ is a principal ideal domain).
\end{proof}

\begin{rmk}
 There is no obvious $\ZZ_2[s]$-module structure on the full $SH(W)$. One possible definition would be to use a non-canonical isomorphism
 \[
  SH(W) \cong \im(h) \oplus SH(W)/\im(h)
 \]
 and extend the module structure from $SH(W)/\im(h)$ to $SH(W)$, e.g.\ by $(s^k,x)\mapsto 0$ for $x\in\im(h)$. However, any such module cannot be torsion-free, simply because in many examples (e.g.\ many Brieskorn manifolds)
 \[
  \dim_{\ZZ_2}(SH_0(W)) > \dim_{\ZZ_2}(SH_{2I(g)}(W)).
 \]
\end{rmk}

\section{Examples: Brieskorn manifolds} \label{sec_Brieskorn_examples}

\subsection{General observations}

To better understand the structure of $\check{SH}(\Sigma)$ in view of concrete examples, let us use a suitable chain complex.
To a specific Hamiltonian $H_{\mu_1, \mu_2}$ as in Figure\ \ref{fig:V_shaped_H} (which only depends on the radial coordinate $r$), we can associate a Morse--Bott chain complex $CF_*^{(a,b)}(H_{\mu_1, \mu_2})$.
In such a complex, the chains are given by pairs $(\mcN_T, \eta)$, where 
\[
 \mcN_T\defeq \{z\in\Sigma \mid \phi_T(z)=z \}
\] 
($\phi_t$ denoting the Reeb flow) is the critical submanifold of $\Sigma$ consisting of periodic Reeb orbits of length $T$, and $\eta$ is a critical point of a Morse function on $\mcN_T$. 
While this complex only contains Reeb orbits of length between $a$ and $b$, these bounds can be chosen arbitrarily large, so this does not really restrict computations in specific examples.

As the Reeb flow on $\Sigma$ is periodic with period normalized to one, we get $\mcN_T\cong \mcN_{T+1}$. Applying the map $S_g$ on this chain complex, we see that the whole critical submanifold $\mcN_T$ gets mapped to $\mcN_{T+1}$. Moreover, if we choose the same Morse functions on $\mcN_T$ and $\mcN_{T+1}$, each generator from $\mcN_T$ gets mapped under $S_g$ to the corresponding generator on $\mcN_{T+1}$, i.e.\ 
\[
 (\mcN_T, \eta) \stackrel{S_g}{\longmapsto} (\mcN_{T+1}, \eta),
\]
although we should be aware that the left- and right hand side belong to different chain complexes, as the Hamiltonian has also changed under $S_g$. Still, in homology and after taking the limits from \eqref{eq_lim_SH_1} and \eqref{eq_lim_SH_2}, we get
\begin{equation*}
 S_g\left([\mcN_T, \eta]\right) = [\mcN_{T+1}, \eta].
\end{equation*}
Together with \eqref{eq_Sg_prod_s}, this gives
\begin{equation} \label{eq_MB_prod_s}
 s \cdot [\mcN_T, \eta] = S_g([\mcN_T, \eta]) = [\mcN_{T+1}, \eta]
\end{equation}
(and similarly for homology classes that are represented by a finite sum of chains).
As the unit of $\check{SH}(\Sigma)$ (which corresponds to the unit of $H^*(\Sigma)$ under the isomorphism $\check{SH}^{(-\epsilon,\epsilon)}\cong H^*(\Sigma)$) is given by the maximum\footnote{Whether it is the minimum or the maximum is a matter of convention.}
on $\mcN_0 \cong \Sigma$, equation~\eqref{eq_MB_prod_s} says in particular that
\begin{equation} \label{eq_MB_s_equals}
 s = S_g([\mcN_0, \max]) = [\mcN_1, \max].
\end{equation}

\subsection{Recollections on Brieskorn manifolds}

For a Brieskorn manifold 
\[
 \Sigma = \Sigma(a_0, \ldots, a_n) \defeq \{z\in \CC^{n+1} \mid z_0^{a_0} + \cdots + z_n^{a_n} = 0, \|z\|=1 \}
\]
with the canonical contact structure $\xi = \ker(\alpha)$, the Reeb flow is given by
\[
 \phi_t(z) = \left(e^{4it/a_0}z_0, \ldots, e^{4it/a_n}z_n\right),
\]
so it is periodic with period $T_P \defeq \lcm_j(a_j) \cdot \frac{\pi}{2}$. So we define the $S^1$-action of Hamiltonian diffeomorphisms on $\Sigma\times\RR$ as
\[
 g_t(z,r) \defeq (\phi_{t\cdot T_P}(z),r) = \left(\left(e^{2\pi it L_P/a_0}z_0, \ldots, e^{2\pi it L_P/a_n}z_n\right), r\right),
\]
where we abbreviated $L_P \defeq \lcm_j(a_j)$. 
To compute the Maslov index $I(g)$, first note that the linearization 
\[
dg_t \colon T(\RR_+\times \Sigma) \to T(\RR_+\times \Sigma)
\]
is the identity on $\Span(R_\alpha, \dd_r)$. Hence, we can use a trivialization of the bundle $\xi$ instead of $T(\RR_+\times \Sigma)$.
Also, the Maslov index is additive under direct sums, so we can use the decomposition $T_z\CC^{n+1} = \xi \oplus \xi^\omega$ of the ambient tangent space $T_z\CC^{n+1}$ of a point $z\in \Sigma$ into the contact distribution $\xi$ and its symplectic complement $\xi^\omega$. 
With the obvious extension of the Reeb flow to $\CC^{n+1}$, the linearization $dg_t$ on the ambient tangent space is given by
\[
 dg_t = \left(\diag(e^{2\pi i L_P t/a_0}, \ldots, e^{2\pi i L_P t/a_0}), \id \right).
\]
So its determinant is $\det (dg_t) =  e^{2\pi i L_P \sum_j 1/a_j}$ and the degree is $L_P \cdot \sum_j \frac{1}{a_j}$.
On $\xi^\omega$, one can find a suitable basis (see e.g.\ \cite[Section~5.3]{KvK}) in which $dg_t$ is given by
\[
 dg_t|_{\xi^\omega} = 
 \begin{pmatrix}
  e^{2\pi i t L_P} & 0 \\ 0 & 1
 \end{pmatrix},
\]
so the degree is $L_P$. By taking the difference, we see that
\begin{equation} \label{eq_Maslov_Brieskorn}
 I(g) = L_P \cdot \left(\sum_{j=0}^n \frac{1}{a_j} - 1\right).
\end{equation}

\subsection{Computing the degrees}

As a consistency check, let us verify that all the degrees in $\ZZ_2[s,s^{-1}]$ actually appear in the chain complex.
We use the grading by the ``product degree''
\[
 \mu_\product = \mu - n,
\]
which is preserved by the product. 
In this grading, the generator $s$ has degree $2I(g)$. So the degrees appearing in $\check{SH}(\Sigma)$ are a finite collection of integers, together with all shifts by multiples of $2I(g)$.

By \eqref{eq_MB_s_equals}, $s = [\mcN_{T_P}, \max]$, i.e.\ the maximum of a Morse function on the critical submanifold $\mcN_{T_P}\cong \Sigma$.
To see that the degrees coincide, we compute
\begin{align}
 \mu_P & \defeq \mu_\product([\mcN_{T_P}, \max]) = \mu_\RS(\mcN_{T_P}) + \dim(N_{T_P}) - \frac{1}{2} (\dim(\mcN_{T_P})-1) - n \nonumber \\
 &= \sum_{j=0}^n \left(\floor*{\frac{L_P}{a_j}}+\ceil*{\frac{L_P}{a_j}}\right) - 2L_P + (2n-1) - (n-1) - n \nonumber \\
 &= 2 \sum_{j=0}^n \frac{L_P}{a_j} - 2 L_P \nonumber \\
 &= 2 L_P \left(\sum_{j=0}^n \frac{1}{a_j} - 1\right) \nonumber \\
 &= 2 I(g), \nonumber
\end{align}
which, as expected, equals the degree of $s$.

Furthermore, let $[\mcN_T, \eta]$ be any generator of $SH$, i.e.\ $\eta$ is a critical point of a Morse function on $\mcN_T$. As $\mcN_{T+T_P}\cong \mcN_T$, we can use the same Morse function on $\mcN_{T+T_P}$ and get a corresponding generator $[\mcN_{T+T_P},\eta]$. According to \eqref{eq_MB_prod_s}, the degrees of $[\mcN_{T+T_P},\eta]$ and $s\cdot [\mcN_T, \eta]$ should match, i.e.\ 
\begin{equation} \label{eq_degree_check}
 \mu_\product ([\mcN_{T+T_P}, \eta]) = \mu_\product([\mcN_T, \eta]) + \mu_P.
\end{equation}
To see this, note that the period of any Reeb orbit of $\Sigma$ is a multiple of $\frac{\pi}{2}$, so we can write $T = L\cdot \frac{\pi}{2}$. Then, we can compute
\begin{align*}
 \mu_\RS(\mcN_{T+T_P}) &= \sum_{j=0}^n \left(\floor*{\frac{L+L_P}{a_j}}+\ceil*{\frac{L+L_P}{a_j}}\right) - 2(L+L_P) \\
 &= \sum_{j=0}^n \left(\floor*{\frac{L}{a_j}}+\ceil*{\frac{L}{a_j}}\right) - 2L + 2 \sum_{j=0}^n \frac{L_P}{a_j} - 2L_P \\
 &= \mu_\RS(\mcN_T) + \mu_P.
\end{align*}
The other terms in the degree formula are the same for $[\mcN_T,\eta]$ and $[\mcN_{T+T_P},\eta]$, thus \eqref{eq_degree_check} is verified.

\begin{example}
 In \cite{Uebele}, symplectic homology was computed for the specific example 
 \[
  \Sigma_\ell \defeq \Sigma(2\ell,2,2,2), \quad \ell\geq 1
 \]
 (and more generally $\Sigma(2\ell, 2,\ldots, 2)$ for $n\geq 3$ odd). While the focus in \cite{Uebele} was on positive symplectic homology $SH^+$, the same methods work for computing $\check{SH}(\Sigma_\ell)$.
 The result can be stated as
 \begin{equation*}
   \check{SH}(\Sigma_\ell) \cong
 \left\{
 \begin{aligned}
  &\ZZ_2 & \qquad & \mbox{if } k=(2\ell+2)N+j \mbox{ for any }N\in\ZZ, j\in\{-1,0,1,2\} \\
  &(\ZZ_2)^2 & \qquad & \mbox{else.}
  \end{aligned}
 \right.
 \end{equation*}
 Note also that $\Sigma_\ell$ is index-positive, hence Theorem~\ref{thm_summary_RFH} can be applied. 
 The index shift is 
 \[
  2I(g) = 4\ell\cdot \left(\frac{1}{2\ell} + \frac{3}{2} -1 \right) = 2\ell + 2,
 \]
 which matches the periodicity of $\check{SH}(\Sigma_\ell)$. Thus, counting the number of generators in one period, we see that $\check{SH}(\Sigma_\ell)$ is a $\ZZ_2[s,s^{-1}]$-module of dimension
 \[
  \dim_{\ZZ_2[s,s^{-1}]} \left( \check{SH}(\Sigma_\ell) \right) = 4\ell.
 \]
\end{example}

\begin{rmk}
 It is tempting to think that this dimension (or the degree of the principal orbit) can distinguish the contact structures of Brieskorn manifolds with different exponents. After all, by Corollary~\ref{cor_prod_in_sympl}, $\check{SH}$ and its product structure depend only on the contact manifold $\Sigma$ (at least under the assumption that $\Sigma$ is product-index-positive, but by Proposition~\ref{prop_weaker_assumption}, the statements about the module structure hold more generally). In this way, one might for instance try to distinguish the contact structures on $\Sigma(\ell p, p, 2,2)$ for fixed $p\in\NN$ and different values of $\ell$, see \cite[Section~3.6]{Uebele}.
 
 However, there is a fundamental difficulty: Since the principal orbit might be itself a power of another generator, the module structure is not uniquely determined. Hence, to distinguish the contact manifolds $\Sigma$ and $\Sigma'$ whose principal orbits have degrees $\mu_P$ and $\mu_P'$, respectively, one would have to exclude the possibility that $\check{SH}(\Sigma)$ is a free module over the Laurent polynomials in a variable $s$ whose degree is a common divisor of $\mu_P$ and $\mu_P'$ (e.g.\ by seeing that $\check{SH}(\Sigma)$ does not have this periodicity). For the example $\Sigma(\ell p, p, 2,2)$, this is probably not possible without explicitly computing some differentials. 
\end{rmk}

\subsection{Comparison with known examples} \label{sec_comparison}

\subsubsection{Cotangent bundles of spheres}

The $(2n-1)$-dimensional Brieskorn manifold $\Sigma(2,\ldots, 2)$ is contactomorphic to the unit cotangent bundle $S^*S^n$ of $S^n$, and its standard filling $W$ is symplectomorphic to $D^*S^n$. Hence, by a famous theorem first proven by Viterbo \cite{Viterbo_part_2}, its symplectic homology is isomorphic to the homology of the free loop space $L S^n$ of $S^n$,
\begin{equation} \label{eq_SH_cotangent_bundle}
 SH_*(D^*S^n; \ZZ) \cong H_*(L S^n; \ZZ).
\end{equation}
Moreover, by \cite{Abb_Schwarz}, the pair-of-pants product on $SH_*(D^*S^n)$ corresponds to the Chas--Sullivan product on $H_*(L S^n)$.
(Note that since $S^n$ is spin, a later correction to this theorem from \cite{Kragh} does not apply here.) 
The right-hand side of \eqref{eq_SH_cotangent_bundle} was computed in \cite{CJY}. Making the degree shift
\[
 \HH_*(L M; \ZZ) \defeq H_{*+n}(L M; \ZZ)
\]
in order for the product to have degree zero, their results can be stated as follows. For $n$ even, 
\begin{equation} \label{eq_string_top_even}
 \HH_*(L S^n; \ZZ) = \Lambda[b] \otimes \ZZ[a,v] / (a^2, ab, 2av),
\end{equation}
where $\Lambda[b]$ denotes the exterior algebra and the degrees of the variables are $|b|=-1$, $|a|=-n$ and $|v|=2n-2$. For $n>1$ odd,
\begin{equation} \label{eq_string_top_odd}
 \HH_*(L S^n; \ZZ) = \Lambda[a] \otimes \ZZ[u],
\end{equation}
where $|a|=-n$ and $|u|=n-1$. 
However, if we take $\ZZ_2$-coefficients, it follows easily from the proof given in \cite{CJY} that for any $n\geq 0$ (even or odd),
\begin{equation} \label{eq_string_top_Z2}
 \HH_*(L S^n; \ZZ_2) = \ZZ_2[a,u] / (a^2),
\end{equation}
with $|a|=-n$ and $|u|=n-1$. 

To compare with Theorem~\ref{thm_summary_RFH}, we need to apply the map $f$ from \eqref{eq_map_SH_to_RFH}. 

\begin{claim} \label{claim_h_inj}
 For $(W,\Sigma) = (D^*S^n, S^*S^n)$ and with $\ZZ_2$-coefficients, the map $f\colon SH(W) \to \check{SH}(\Sigma)$ is injective.
\end{claim}

\begin{proof}
 By exactness of the sequence \eqref{eq_les_SH_RFH}, it suffices to show that the map $h$ from \eqref{eq_h_factoring} vanishes. 
 For this, in turn, it suffices to show that the map 
 \[
  \incl_* \colon H_k(D^*S^n) \longrightarrow H_k(D^*S^n, S^*S^n)
 \]
 vanishes in all degrees. As $D^*S^n\simeq S^n$, $H_k(D^*S^n)$ vanishes for $k\neq 0,n$, and $H_k(D^*S^n, S^*S^n)$ vanishes for $k=0$. Thus, the only non-trivial degree is $k=n$, for which it follows from the long exact sequence of the pair $(D^*S^n, S^*S^n)$ with $\ZZ_2$-coefficients
 \[
  \cdots \longrightarrow H_n(D^*S^n) \longrightarrow H_n(D^*S^n, S^*S^n) \longrightarrow H_{n-1}(S^*S^n) \longrightarrow H_{n-1}(D^*S^n) \longrightarrow \cdots
 \]
and $H_{n-1}(S^*S^n) \cong \ZZ_2$, $H_{n-1}(D^*S^n)=0$ that
 \[
  \incl_* \colon H_n(D^*S^n)\cong \ZZ_2 \longrightarrow H_n(D^*S^n, S^*S^n)\cong \ZZ_2
 \]
 is the zero map.
\end{proof}

\begin{rmk}
 For $n$ odd, Claim~\ref{claim_h_inj} is also true for $\ZZ$-coefficients, while for $n$ even, the last step in the proof only works over $\ZZ_2$.
\end{rmk}

Now, we compare with $\check{SH}(\Sigma(2,\ldots, 2))$. Note that all critical manifolds are of the form $\mcN_{N\pi}$ for $N\in \ZZ$, hence they are diffeomorphic to $\Sigma = \Sigma(2,\ldots,2)$. The degree of a generator $[\mcN_{N\pi}, \eta]$, in the product grading, can be computed as
\begin{align*}
 \mu_\product([\mcN_{N\pi}, \eta]) &= \mu_\RS(\mcN_{N\pi}) - \frac{1}{2} (\dim(\Sigma) -1) + \ind_\Morse(\eta) - n \\
 &= \sum_{j=0}^n \left(\floor{N} + \ceil{N} \right) - 4N - (n-1) + \ind_\Morse(\eta) - n \\
 &= 2N (n-1) -2n +1 + \ind_\Morse(\eta),
\end{align*}
and if we choose a perfect Morse function on $\Sigma\cong S^*S^n$, $\ind_\Morse(\eta) \in\{0,n-1,n,2n-1\}$.
Also note that all generators  with $N>0$, corresponding to positively oriented Reeb orbits, have Conley--Zehnder index at least $n-1$, from which it follows that $\Sigma$ is index-positive for $n\geq 3$.
As for the differential, it turns out that, at least for $n\geq 3$, all differentials of this chain complex vanish. For $n\geq 4$, this follows immediately for degree and action reasons, while for $n=3$, it is a special case of the computations done in \cite{Uebele}.

Hence, as a $\ZZ_2$-vector space, the $\bigvee$-shaped symplectic homology of $\Sigma$ is given by
\begin{equation} \label{eq_SH_for_Sigma_222}
 \check{SH}(\Sigma) \cong 
  \left\{
 \begin{aligned}
  \ZZ_2 & \qquad \mbox{if } k = 2N(n-1) & \\
  & \qquad\mbox{or } k = 2N(n-1) -n+1 & \\
  & \qquad\mbox{or } k = 2N(n-1) -n & \\
  & \qquad\mbox{or } k = 2N(n-1) -2n+1 & & \mbox{for some }N\in\ZZ, \\
  0 & \qquad\mbox{else.}
  \end{aligned}
 \right.
\end{equation}

It can easily be checked that these degrees with $N\geq 0$ match those in \eqref{eq_string_top_Z2}, in accordance with \eqref{eq_SH_cotangent_bundle} and Claim~\ref{claim_h_inj}. Moreover, the generator $s = S_g(1) = [\mcN_1, \max]$ appears in the first line of \eqref{eq_SH_for_Sigma_222} with $N=1$.

Now, the main point in the comparison concerns the product structure.
Theorem~\ref{thm_summary_RFH} says that $\check{SH}(\Sigma)$ is a free module over $\ZZ_2[s,s^{-1}]$, with the module structure given by the pair-of-pants product. This matches with \eqref{eq_string_top_Z2}, where $s$ corresponds to $u^2$.

However, Theorem~\ref{thm_summary_RFH} does not see that $s$ has a square root.
Instead, we only see that $\check{SH}(\Sigma)$ is a four-dimensional free module over $\ZZ_2[s,s^{-1}]$, with the first four lines in \eqref{eq_SH_for_Sigma_222} each giving a generator.
This implies that as an algebra, $\check{SH}(\Sigma)$ can be generated by at most four elements, while \eqref{eq_SH_cotangent_bundle} and \eqref{eq_string_top_Z2} show that two generators suffice.

As an interesting side note, Theorem~\ref{thm_summary_RFH} in combination with \eqref{eq_string_top_Z2} and Lemma~\ref{lem_f_intertwines_product} reveals the full ring structure on $\check{SH}(\Sigma)$:

\begin{thm} \label{thm_full_ring_str}
 The ring structure of $\check{SH}(S^*S^n)$ for $n\geq 3$ is given by
 \begin{equation} \label{eq_RFH_ring_Sigma_222}
  \check{SH}(S^*S^n) = \ZZ_2[a,u,u^{-1}] / (a^2),
 \end{equation}
 where $|a|=-n$ and $|u|=n-1$.
\end{thm}

\begin{proof}
 As a $\ZZ_2$-vector space, this follows from \eqref{eq_SH_for_Sigma_222}. So it remains to show that the product matches, i.e.\ that the expressions $\langle x\cdot y, z\rangle$ are what \eqref{eq_RFH_ring_Sigma_222} predicts.
 
 To see this, note that for any $x, y\in \check{SH}(\Sigma)$, we can find an $N\geq 0$ such that $s^N\cdot x, s^N\cdot y \in \im(f)$. (Here it is important that we use $\check{SH}$ and not $\widetilde{SH}$.)
 Now we can compute
 \[
  \langle x\cdot y, z\rangle = \langle S_g^{2N}(x\cdot y), S_g^{2N}(z)\rangle = \langle (s^N \cdot x)\cdot (s^N\cdot y), S_g^{2N}(z)\rangle ,
 \]
 and the right hand side only involves terms in $\im(f)$. For those, we already know from \eqref{eq_string_top_Z2} that the product structure is the one predicted by \eqref{eq_RFH_ring_Sigma_222}.
\end{proof}

Note that by \cite[Theorem~1.10]{Ciel_Fra_Oan}, there is an isomorphism
\begin{equation} \label{eq_iso_loop_cohom}
 \check{SH}_k(D^*M) \cong H^{-k+n+1}(LM) \qquad \text{for } k<n
\end{equation}
between $\check{SH}$ of the cotangent bundle and the cohomology of the free loop space of $M$ in sufficiently negative degrees (in the product grading). On this part, the pair-of-pants product is conjectured to be related to the Goresky--Hingston product on $H^*(LM, L_0 M)$ (the cohomology of the free loop space, relative to constant loops). Indeed, if we restrict the degrees further to the range where $H^*(LM, L_0 M) \cong H^*(LM)$ (i.e.\ $*>n+1$), these products might actually coincide. 

For spheres, the Goresky--Hingston product has been computed in \cite{Gor-Hing}. With $\ZZ_2$-coefficients and up to a grading shift, the result is
\begin{equation*}
 H^*(LS^n, L_0 S^n) \cong \Lambda(U) \otimes \ZZ_2[T]_{\geq 2},
\end{equation*}
where $\deg(T)=n-1$, $\deg(U)=1$ and $\ZZ_2[T]_{\geq 2}$ denotes the ideal in $\ZZ_2[T]$ generated by $T^2$. Thus, this example supports the conjecture that the product coincides with the pair-of-pants product on \eqref{eq_RFH_ring_Sigma_222}, with the identification $T\mapsto u^{-1}$ and $U\mapsto au$.

\subsubsection{$A_k$-surface singularities} \label{sec_A_k_surface}

Besides cotangent bundles, the only example of Brieskorn manifolds for which the product structure on symplectic homology has been computed are the $A_k$-surface singularities. They are by definition the fillings of the Brieskorn manifolds
\[
 \Sigma(k+1, 2, 2) \cong L(k+1,k)
\]
for $k>1$, which are contactomorphic to the lens spaces $L(k,k+1)$. The symplectic homology of their canonical filling, along with its ring structure has been computed in \cite{EL} (although it should be mentioned that their methods rely on theorems from \cite{BEE}, which are note yet proven in full rigor). The following theorem specializes the results of \cite[Theorem~40]{EL} to $\ZZ_2$-coefficients.

\begin{thm}[\cite{EL}] \label{thm_Etgü-Lekili}
 Denote by $W_k$ the canonical filling of $\Sigma(k+1,2,2)$. For $k$ even, its symplectic homology is given by
 \begin{equation} \label{eq_SH_Ak_sing_k_even}
  SH(W_k) = \ZZ_2[s_1,\ldots, s_k, t_1,t_0,t_{-2}] \Big/ (s_i s_j=0,\ s_i t_j=0,\ t_1^2=0,\ t_0^k=0),
 \end{equation}
 where the degrees are $|s_i|=-2$, $|t_1|=-1$, $|t_0|=0$ and $t_{-2}=2$. For $k$ odd,
 \begin{align}
  SH(W_k) = \ZZ_2[s_1,\ldots, s_k, t_1,t_0,u_{-1},t_{-2}] \Big/ & (s_i s_j=0,\ s_i t_1=0,\ s_i t_0=0,\ t_1^2=0,  \nonumber\\
  & s_iu_{-1}=t_1t_0^{k-1},\ s_it_{-2}=t_0^k,\ t_0u_{-1}=t_1t_{-2}, \label{eq_SH_Ak_sing_k_odd} \\
  & t_1u_{-1}=\alpha t_0^k,\ u_{-1}^2=\beta t_0^{n-1}), \nonumber
 \end{align}
 where the degrees are $|s_i|=-2$, $|t_1|=-1$, $|t_0|=0$, $u_{-1}=1$ and $t_{-2}=2$, and $\alpha=\beta=1$ if $4|(k+1)$, otherwise $\alpha=\beta=0$.
\end{thm}

Here, the gradings are defined via filling disks in $W_k$, which is simply-connected. Note that, due to different conventions, our grading differs from \cite{EL} by a minus sign.

Unfortunately, $\Sigma(k+1,2,2)$ is not index-positive, because there are Reeb orbits with Conley--Zehnder index one (and which represent non-trivial classes in contact homology, so taking another contact form does not help). Thus, Theorem~\ref{thm_summary_RFH} cannot really be applied.
However, as far as one can infer from $SH(W_k)$, its conclusion still seems to holds.
For $k$ even, the grading shift is
\[
 \mu_P = 4(k+1) \left(\sum_j \frac{1}{a_j} - 1 \right) = 4,
\]
so it suffices to see that there is a generator of degree four whose products make $\check{SH}(\Sigma)$ periodic. Indeed, $f(t_{-2})^2$ has degree four. Moreover, it follows from \eqref{eq_h_factoring} and exactness of \eqref{eq_les_SH_RFH} that all $s_i$ get divided out by $f$. Hence, in $\im(f)\subset \check{SH}(\Sigma)$, there is no relation involving $f(t_{-2})$ (or its square), thus periodicity holds.

For $k$ odd, the grading shift is
\[
 \mu_P = 2(k+1) \left(\sum_j \frac{1}{a_j} - 1 \right) = 2,
\]
so the generator corresponding to the principle orbit could be $f(t_2)$ directly. The ring structure is more complicated in this case, but it still turns out that none of the relations in \eqref{eq_SH_Ak_sing_k_odd} destroys the periodicity coming from multiplication by $f(t_{-2})$.

In light of this result, it seems reasonable to conjecture that the conclusion of Theorem~\ref{thm_summary_RFH} holds for Brieskorn manifolds in general, even if they are not index-positive.

\end{document}